\documentclass[12pt,reqno]{amsart}
\usepackage{amssymb}
\usepackage{amsfonts}
\usepackage{hyperref}

\theoremstyle{plain}

\newtheorem{algorithm}{Algorithm}[section]
\newtheorem{thm}{Thm}

\newtheorem{examplelet}[thm]{Example}

\newtheorem{lemma}[algorithm]{Lemma}

\newtheorem{theorem} [algorithm] {Theorem}
\newtheorem{theoremlet}[thm]{Theorem}

\newtheorem{definitionlet}[thm]{Definition}
\newtheorem*{remarknonum}{Remark}

\newtheorem{proposition}[algorithm]{Proposition}

\numberwithin{equation}{algorithm}

\input{tcilatex}

\begin{document}
\title{The Sub-Index of Critical Points of Distance Functions}
\author{Barbara Herzog }
\address{Department of Mathematics, Doane College}
\email{barbara.herzog@doane.edu}
\urladdr{http://www.doane.edu/barbara-herzog}
\author{Frederick Wilhelm}
\address{Department of Mathematics, Univ.of Calf., Riverside}
\email{fred@math.ucr.edu }
\urladdr{https://sites.google.com/site/frederickhwilhelmjr/}
\date{August 28, 2014}

\begin{abstract}
We define a new notion---the sub-index of a critical point of a distance
function. We show how sub-index affects the homotopy type of sublevel sets
of distance functions.
\end{abstract}

\maketitle

Morse Theory is based on the idea that a smooth function on a manifold
yields data about its topology. Specifically, Morse's Isotopy Lemma tells us
that two sublevels are diffeomorphic provided there are no critical points
between their corresponding levels. Further, the index of the Hessian of the
function constrains the change in homotopy type caused by a critical point.

Since Riemannian distance functions are not smooth everywhere, critical
points and the Hessian cannot be defined in the usual way. In 1977, Grove
and Shiohama created a definition of critical point for distance functions
and used it to generalize Morse's Isotopy Lemma to this case. Their
generalization had a profound impact on Riemannian geometry. However,
without a definition of index, the remainder of Morse Theory cannot be
generalized. Here we propose a partial remedy to this situation. Before
stating it, we recall the definition of critical points.

\begin{definitionlet}
Let $M$ be a complete Riemannian $n$--manifold. For $x_{0}\in M,$ let $%
S_{x_{0}}$ be the unit tangent sphere at $x_{0}$ and $K\subset M$ be
compact. Set 
\begin{equation*}
\Uparrow _{x_{0}}^{K}\equiv \left\{ w\in S_{x_{0}}\ |\ w\text{ is tangent at 
}x_{0}\text{ to a minimal geodesic from }x_{0}\text{ to }K\right\} .
\end{equation*}%
A point $x_{0}$ in $M$ is \emph{regular} for $\mathrm{dist}\left( K,\cdot
\right) $ if there exists a $v$ in $T_{x_{0}}M$ such that $\sphericalangle
(v,\Uparrow _{x_{0}}^{K})>\frac{\pi }{2}$. Otherwise, $x_{0}$ is called 
\emph{critical} for $\mathrm{dist}\left( K,\cdot \right) .$
\end{definitionlet}

Since the set $\Uparrow _{x_{0}}^{K}$ can be quite unwieldy, for critical
points $x_{0}$ of $\mathrm{dist}\left( K,\cdot \right) $ we consider%
\begin{eqnarray*}
A\left( \Uparrow _{x_{0}}^{K}\right) &\equiv &\left\{ v\in
S_{x_{0}}|\sphericalangle \left( v,\Uparrow _{x_{0}}^{K}\right) \geq \frac{%
\pi }{2}\right\} \vspace*{0.05in} \\
&=&\left\{ v\in S_{x_{0}}|\sphericalangle \left( v,\Uparrow
_{x_{0}}^{K}\right) =\frac{\pi }{2}\right\} .
\end{eqnarray*}

Since $A(\Uparrow _{x_{0}}^{K})$ is an intersection of hemispheres in $%
S_{x_{0}},$ $A(\Uparrow _{x_{0}}^{K})$ is convex. In particular, if $%
\partial A(\Uparrow _{x_{0}}^{K})=\emptyset ,$ then $A(\Uparrow
_{x_{0}}^{K}) $ is a great subsphere of $S_{x_{0}}.$ Motivated by these
observations we define sub-index as follows.

\begin{definitionlet}
Let $x_{0}$ be a critical point of $\mathrm{dist}\left( K,\cdot \right) .$
The sub-index of $x_{0}$ is 
\begin{equation*}
\begin{array}{ll}
n & \hspace*{0.25in}\text{if }A(\Uparrow _{x_{0}}^{K})=\emptyset \vspace*{%
0.05in} \\ 
n-\dim \mathrm{span}\left\{ A(\Uparrow _{x_{0}}^{K})\right\} & \hspace*{%
0.25in}\text{if }A(\Uparrow _{x_{0}}^{K})\neq \emptyset \text{ and }\partial
A(\Uparrow _{x_{0}}^{K})=\emptyset \vspace*{0.05in} \\ 
\infty & \hspace*{0.25in}\text{if }\partial A(\Uparrow _{x_{0}}^{K})\neq
\emptyset .%
\end{array}%
\end{equation*}
\end{definitionlet}

For $r\in \mathbb{R},$ we set 
\begin{equation*}
B\left( K,r\right) \equiv \left\{ \left. x\in M\right\vert \text{ \textrm{%
dist}}\left( K,x\right) <r\right\} .
\end{equation*}

\begin{theoremlet}
\label{mainthm}Suppose that the critical points for $\mathrm{dist}\left(
K,\cdot \right) $ are isolated and that those with $\mathrm{dist}\left(
K,x_{0}\right) =c_{0}$ have sub-index $\geq \lambda .$ Then for all
sufficiently small $\varepsilon >0$, the inclusion $B\left(
K,c_{0}-\varepsilon \right) \hookrightarrow B\left( K,c_{0}+\varepsilon
\right) $ is $(\lambda -1)$-connected, that is, 
\begin{equation*}
\pi _{i}(B\left( K,c_{0}+\varepsilon \right) ,\text{ }B\left(
K,c_{0}-\varepsilon \right) )=0
\end{equation*}%
for $i=1,\ldots ,(\lambda -1)$.
\end{theoremlet}

\begin{remarknonum}
When all critical points with $\mathrm{dist}\left( K,x_{0}\right) =c_{0}$
have $\partial A(\Uparrow _{x_{0}}^{K})\neq \emptyset ,$ $\lambda $ is $%
\infty ,$ and the theorem asserts that $B\left( K,c_{0}-\varepsilon \right)
\hookrightarrow B\left( K,c_{0}+\varepsilon \right) $ is a weak homotopy
equivalence.
\end{remarknonum}

This extends an aspect of the classical result of Morse Theory about
nondegenerate critical points of smooth functions $f:M\longrightarrow 
\mathbb{R}.$ To see why, set $M^{r}\equiv \left\{ x\in M|f\left( x\right)
\leq r\right\} .$ According to the classical result of Morse Theory, if $%
x_{0}$ is a nondegenerate critical point of $f$ of index $\lambda $, then $%
M^{r+\varepsilon }$ has the homotopy type of a $CW$--complex obtained from $%
M^{r-\varepsilon }$ by attaching a $\lambda $--cell. This implies that%
\begin{equation*}
\pi _{i}(M^{r+\varepsilon },\text{ }M^{r-\varepsilon })=0\text{ for all }%
i\leq \lambda -1.
\end{equation*}%
Our theorem recovers this aspect of Morse Theory. On the other hand, in the
case of a smooth function, 
\begin{equation*}
H_{i}(M^{r+\varepsilon },\text{ }M^{r-\varepsilon })=0\text{ for all }i\neq
\lambda .
\end{equation*}%
We have no analogous result about relative homology in dimensions larger
than $\lambda $ for distance functions. For this reason, we call our notion
sub-index rather than index.

Our hypothesis that the critical points for $\mathrm{dist}\left( K,\cdot
\right) $ be isolated is implicitly present in the classical result, since
nondegenerate critical points of smooth functions are isolated.

Our theory yields a very rigid structure for critical points that impact the
fundamental group.

\begin{theoremlet}
\label{thm2}Suppose that the critical points for $\mathrm{dist}\left(
K,\cdot \right) $ are isolated and that for some $c_{0}>0$ and all
sufficiently small $\varepsilon >0,$ 
\begin{equation*}
\pi _{1}(B\left( K,c_{0}+\varepsilon \right) ,B\left( K,c_{0}-\varepsilon
\right) )\neq 0.
\end{equation*}

Then there is a critical point $x_{0}$ for $\mathrm{dist}\left( K,\cdot
\right) $ with $\mathrm{dist}\left( K,x_{0}\right) =c_{0}$ so that there are
only two minimal geodesics from $K$ to $x_{0}$ that make angle $\pi $ at $%
x_{0}$. Moreover, the ends of these geodesic segments are not conjugate
along the segments.
\end{theoremlet}

The theory of sub-index is beautifully exemplified by flat tori.

\begin{examplelet}
\label{torus}Let $M$ be the flat $n$-torus obtained from the standard
embedding of $\mathbb{Z}^{n}\hookrightarrow \mathbb{R}^{n},$ i.e. that with
fundamental domain $\left[ 0,1\right] ^{n}.$ For $K$ we take the equivalence
class of the point $\left( \frac{1}{2},\frac{1}{2},\ldots ,\frac{1}{2}%
\right) .$ The cut locus of $K$ is the equivalence class of the boundary of $%
\left[ 0,1\right] ^{n}.$ The critical points are the centers of subcubes of
the boundary of $\left[ 0,1\right] ^{n}.$ For example, the center of the $k$%
--dimensional subcube 
\begin{equation*}
\left[ 0,1\right] ^{k}\times \overset{\left( n-k\right) \text{ times }}{%
\overbrace{\left( 0,0,0,\ldots ,0\right) }}
\end{equation*}%
is critical. The sub-index of a critical point at the center of a $k$%
--dimensional subcube is $n-k.$ For instance, the equivalence class of the
corners $\left[ \left( 1,1,1,\ldots ,1\right) \right] $ is a maximum and has
sub-index $n.$ In light of the conjecture that the total Betti numbers of a
nonnegatively curved $n$--manifold is $\leq 2^{n}=\Sigma _{i}Betti_{i}\left(
T^{n}\right) ,$ it is intriguing that, for this example, the number of
critical points of sub-index $\lambda $ coincides with the $\lambda ^{th}$%
--Betti number of the torus.

A slide show illustrating this example in dimension 3 can be found at \cite%
{Herz}.
\end{examplelet}

The proof of Theorem \ref{mainthm} is divided into three cases:%
\begin{equation*}
\begin{array}{ll}
\text{Case 1}\text{:} & A(\Uparrow _{x_{0}}^{K})=\emptyset ,\vspace*{0.05in}
\\ 
\text{Case 2}\text{:} & A(\Uparrow _{x_{0}}^{K})\neq \emptyset \text{ and }%
\partial A(\Uparrow _{x_{0}}^{K})=\emptyset ,\vspace*{0.05in} \\ 
\text{Case 3}\text{: } & \partial A(\Uparrow _{x_{0}}^{K})\neq \emptyset .%
\end{array}%
\end{equation*}%
In Section 1, we establish notations and conventions. A lemma that reduces
the proof of Theorem \ref{mainthm} to a local problem is presented in
Section \ref{local reduction section}. In Section \ref{flow section}, we
study certain flows on $\mathbb{R}^{n}$ and $S^{n-1}\subset \mathbb{R}^{n}$
that we then transfer to $M$ via normal coordinates, in Section \ref%
{proofofresults}, where the proof of Theorem \ref{mainthm} is completed.

For a general idea of the proof, note that if $A(\Uparrow _{x_{0}}^{K})$ is
empty, all vectors along minimal geodesics emanating from $x_{0}$ point in a
direction of decrease for $\mathrm{dist}\left( K,\cdot \right) $. This means 
$x_{0}$ is an isolated local maximum. So any cell of dimension less than $n$
can be deformed into $B\left( K,c_{0}-\varepsilon \right) $.

For the other two cases, $A(\Uparrow _{x_{0}}^{K})$ is not empty, and for $%
k=1,\ldots ,(\lambda -1)$ we consider a $k$-dimensional cell $\iota
:E^{k}\longrightarrow $ $B\left( K,c_{0}+\varepsilon \right) $ with $\iota
\left( \partial E^{k}\right) \subset $ $B\left( K,c_{0}-\varepsilon \right) $%
. To prove the theorem, we construct a homotopy of $\iota $ into $B\left(
K,c_{0}-\varepsilon \right) $ that fixes $\iota |_{\partial E^{k}}$. When
the boundary of $A(\Uparrow _{x_{0}}^{K})$ is empty, $\mathrm{span}\left(
A(\Uparrow _{x_{0}}^{K})\right) \subset T_{x_{0}}M$ is a linear subspace of
dimension $n-\lambda .$ Transversality allows us to move $\iota \left(
E^{k}\right) $ away from $\exp _{x_{0}}\left( \mathrm{span}\left( A(\Uparrow
_{x_{0}}^{K})\right) \right) $.

If both $A(\Uparrow _{x_{0}}^{K})$ and its boundary are not empty, $%
A(\Uparrow _{x_{0}}^{K})$ contains a vector $w_{s}$ such that 
\begin{equation*}
A(\Uparrow _{x_{0}}^{K})\subset \left\{ \left. v\in S_{x_{0}}\text{ }%
\right\vert \text{ }\sphericalangle \left( v,w_{s}\right) \leq \frac{\pi }{2}%
\right\} .
\end{equation*}%
We argue that there is an extension of $-w_{s}$ to a vector field whose flow
moves $\iota \left( E^{k}\right) $ into $B\left( K,c_{0}-\varepsilon \right) 
$.

We prove Theorem \ref{thm2} in Sections \ref{conj section} and \ref{pi_1
sect}, and state alternative versions of Theorems \ref{mainthm} and \ref%
{thm2} in Section \ref{alt section}.

\bigskip

\noindent \textbf{Acknowledgement: }\emph{We are profoundly grateful to the
referee for pointing out that a preliminary draft contained an error in the
proof of Theorem \ref{solu flow thm}.}

\bigskip

\section{Background, Notations, and Conventions\label{background section}}

We assume throughout that $M$ is a complete Riemannian $n$--manifold. We
write $\mathrm{dist}_{K}\left( \cdot \right) $ for \textrm{dist}$\left(
K,\cdot \right) $ and set%
\begin{eqnarray*}
S\left( K,r\right) &\equiv &\left\{ \left. x\in M\text{ }\right\vert \text{ }%
\mathrm{dist}_{K}\left( x\right) =r\right\} \text{ and} \\
B\left( K,r\right) &\equiv &\left\{ \left. x\in M\text{ }\right\vert \text{ }%
\mathrm{dist}_{K}\left( x\right) <r\right\} .
\end{eqnarray*}

We use the terms \emph{segment }and \emph{minimal geodesic }interchangeably.

For $x_{0}\in M$, recall that $S_{x_{0}}$ is the unit tangent sphere at $%
x_{0}.$ For either $K\subset M$ or $K\in M$,%
\begin{equation*}
\Uparrow _{x_{0}}^{K}\equiv \left\{ w\in S_{x_{0}}\ |\ w\text{ is tangent at 
}x_{0}\text{ to a minimal geodesic from }x_{0}\text{ to }K\right\} .
\end{equation*}%
We let 
\begin{equation*}
\uparrow _{x_{0}}^{K}\in \Uparrow _{x_{0}}^{K}
\end{equation*}%
stand for any member of $\Uparrow _{x_{0}}^{K},$ and we let $U_{\theta }$ be
the $\theta $--neighborhood of $\Uparrow _{x_{0}}^{K}$ in $S_{x_{0}},$ that
is 
\begin{equation*}
U_{\theta }\equiv \left\{ \left. v\in S_{x_{0}}\text{ }\right\vert \text{ }%
\sphericalangle \left( v,\Uparrow _{x_{0}}^{K}\right) <\theta \right\} .
\end{equation*}

Throughout we assume that $x_{0}$ is an isolated critical point for $\mathrm{%
dist}_{K}$ with $\mathrm{dist}_{K}\left( x_{0}\right) =c_{0},$ and we denote
the injectivity radius at $x_{0}$ by $\mathrm{inj}_{x_{0}}.$ By \cite{Plaut}%
, $\mathrm{dist}_{K}$ is directionally differentiable, and in a direction $%
v\in S_{x_{0}},$ the derivative is 
\begin{equation*}
D_{v}(\mathrm{dist}_{K})=-\cos \sphericalangle (v,\Uparrow _{x_{0}}^{K}).
\end{equation*}%
An immediate consequence is

\begin{lemma}
\label{Taylorpolylemma}Given $x_{0}\in M$ and $\varepsilon >0$, there exists 
$\rho >0$ such that for all $v\in S_{x_{0}}$ 
\begin{equation*}
c_{0}+\left( -\cos \sphericalangle (v,\Uparrow _{x_{0}}^{K})-\varepsilon
\right) \cdot t\leq \mathrm{dist}_{K}\left( {\exp _{x_{0}}\left( tv\right) }%
\right) \leq c_{0}+\left( -\cos \sphericalangle (v,\Uparrow
_{x_{0}}^{K})+\varepsilon \right) \cdot t
\end{equation*}%
for all $t\in \lbrack 0,\rho ]$.
\end{lemma}

For simplicity, we only discuss the proof of Theorem \ref{mainthm} in the
special case when $x_{0}$ is the only critical point with $\mathrm{dist}%
_{K}\left( x_{0}\right) =c_{0}$. Our method easily adapts to the general
case with minor technical modifications. By compactness, it follows that
there is an $\varepsilon _{0}>0$ such that 
\begin{equation}
x_{0}\text{ is the only critical point for }\mathrm{dist}_{K}\text{ in }%
\overline{B\left( K,c_{0}+\varepsilon _{0}\right) }\setminus B\left(
K,c_{0}-\varepsilon _{0}\right) .  \label{eps_0 inequal}
\end{equation}

For $v,w\in S_{x_{0}}\subset T_{x_{0}}M$ we use $\mathrm{dist}\left(
v,w\right) $ or $\sphericalangle \left( v,w\right) ,$ depending on the
context, and whether we wish to emphasize the importance of $v,w\in
S_{x_{0}} $ or the importance of $v,w\in T_{x_{0}}M.$ A similar comment
applies to $v,w\in S^{n-1}\subset \mathbb{R}^{n}.$ Thus for $v,w\in
S_{x_{0}}\subset T_{x_{0}}M$ or $v,w\in S^{n-1}\subset \mathbb{R}^{n}$ we use%
\begin{equation*}
\sphericalangle \left( v,w\right) =\mathrm{dist}\left( v,w\right)
\end{equation*}%
interchangeably. However, for $v,a,w\in S_{x_{0}}\subset T_{x_{0}}M$ or $%
v,a,w\in S^{n-1}\subset \mathbb{R}^{n}$ 
\begin{equation*}
\sphericalangle \left( v,a,w\right)
\end{equation*}%
refers to the angle of the hinge in $S_{x_{0}}$ or $S^{n-1}$ with vertex $a.$

\section{The Local Reduction Lemma\label{local reduction section}}

\begin{lemma}
\label{localreduction}(Local Reduction Lemma) Let $\varepsilon _{0}$ be as
in Statement \ref{eps_0 inequal}. Suppose that for $r\in \left( 0,\frac{%
\mathrm{inj}_{x_{0}}}{2}\right) $, we have 
\begin{equation*}
\overline{B(x_{0},2r)}\subset B\left( K,c_{0}+\varepsilon _{0}\right)
\setminus \overline{B\left( K,c_{0}-\varepsilon _{0}\right) }.
\end{equation*}%
Then there is an $\eta >0$ and a deformation retract of $\overline{B\left(
K,c_{0}+\varepsilon _{0}\right) }$ to a subset of $U\equiv B\left(
K,c_{0}-\eta \right) \cup B(x_{0},r)$.
\end{lemma}

\begin{proof}
(Cf the proof of Lemma 55 in \cite{Pet}.) Following \cite{Pet}, we say that
for $\alpha \in \left( 0,\frac{\pi }{2}\right) ,$ $x$ is $\alpha $--regular
for $\mathrm{dist}_{K}$ provided there is a $v\in $ $S_{x}$ so that 
\begin{equation*}
\sphericalangle \left( v,w\right) <\alpha \text{ for all }w\in \Uparrow
_{x}^{K}.
\end{equation*}%
From Proposition 47 in \cite{Pet}, we see that the set of $\alpha $--regular
points for $\mathrm{dist}_{K}$ is open and has a unit vector field $X$ so
that for every integral curve $c$ of $X$ 
\begin{equation*}
\frac{\mathrm{dist}_{K}c\left( t\right) -\mathrm{dist}_{K}c\left( s\right) }{%
t-s}<-\cos \alpha <0
\end{equation*}%
for all $t,s$ in the domain of $c.$

Since 
\begin{equation*}
\bar{V}_{0}\equiv \overline{B\left( K,c_{0}+\varepsilon _{0}\right) }%
\setminus \left\{ B\left( K,c_{0}-\varepsilon _{0}\right) \cup B(x_{0},\frac{%
1}{100}r)\right\}
\end{equation*}%
is compact and free of critical points for $\mathrm{dist}_{K},$ it follows
that there is an $\alpha _{0}>0$ so that $\mathrm{dist}_{K}$ is $\alpha _{0}$%
--regular on $\bar{V}_{0}.$

Let $X$ be a unit vector field on $\bar{V}_{0}$ so that%
\begin{equation}
\frac{\mathrm{dist}_{K}c\left( t\right) -\mathrm{dist}_{K}c\left( s\right) }{%
t-s}<-\cos \alpha _{0}<0.  \label{alph_0 reg Inequal}
\end{equation}%
We also have that 
\begin{equation*}
\bar{V}_{1}\equiv \overline{B\left( K,c_{0}+\varepsilon _{0}\right) }%
\setminus \left\{ B\left( K,c_{0}+\frac{\cos \left( \alpha _{0}\right) }{100}%
r\right) \right\}
\end{equation*}%
is compact and free of critical points for $\mathrm{dist}_{K},$ so there is $%
\alpha _{1}>0$ such that $\mathrm{dist}_{K}$ is $\alpha _{1}$--regular on $%
\bar{V}_{1}.$ It follows from the Isotopy Lemma (see, e.g., Lemma 55 in \cite%
{Pet}) that $\overline{B\left( K,c_{0}+\varepsilon _{0}\right) }$ is
homeomorphic to and deformation retracts onto $\overline{B\left( K,c_{0}+%
\frac{\cos \left( \alpha _{0}\right) }{100}r\right) }.$ Let $\Phi $ be the
deformation retraction that accomplishes this, and let $\bar{\Psi}$ be the
flow of $X.$ It follows from Inequality \ref{alph_0 reg Inequal} and the
fact that $X$ is unit, that for all 
\begin{equation*}
y\in B\overline{\left( K,c_{0}+\frac{\cos \left( \alpha _{0}\right) }{100}%
r\right) }\setminus \left\{ B\left( K,c_{0}-\varepsilon _{0}\right) \cup
B(x_{0},\frac{1}{4}r)\right\}
\end{equation*}%
there is a first time $t_{y}\in \left( 0,\frac{r}{50}\right) $ that varies
continuously with $y$ so that 
\begin{equation}
\bar{\Psi}_{t_{y}}\left( y\right) \in \overline{B\left( K,c_{0}-\frac{\cos
\left( \alpha _{0}\right) }{100}r\right) }.  \label{gets below}
\end{equation}

If, in addition, $y\in B(x_{0},\frac{3}{4}r),$ then for all $t\in \left[
0,t_{y}\right] \subset \left[ 0,\frac{r}{50}\right] ,$ 
\begin{equation}
\bar{\Psi}_{t}\left( y\right) \text{ is in }B\left( x_{0},r\right) .
\label{Stays in  R ball}
\end{equation}%
Let $\psi :M\longrightarrow \left[ 0,1\right] $ be $C^{\infty }$ and satisfy 
\begin{equation*}
\psi |_{\overline{B(x_{0},\frac{1}{4}r)}}\equiv 0\text{ and }\psi
|_{M\setminus B(x_{0},\frac{1}{2}r)}\equiv 1.
\end{equation*}%
Set 
\begin{equation*}
\Psi _{t}\left( y\right) =\left\{ 
\begin{array}{ll}
\bar{\Psi}_{t_{y}\cdot \psi \left( y\right) \cdot t}\left( y\right) & \text{%
if }y\in M\setminus B(x_{0},\frac{1}{8}r) \\ 
y & \text{if }y\in \overline{B(x_{0},\frac{1}{4}r)}%
\end{array}%
\right.
\end{equation*}%
and $\eta =\frac{\cos \left( \alpha _{0}\right) }{200}r.$

By combining \ref{gets below} with the definition of $\Psi $ and the fact
that $X$ is unit, it follows that for $y\in B\overline{\left( K,c_{0}+\frac{%
\cos \left( \alpha _{0}\right) }{100}r\right) }\setminus \left\{ B\left(
K,c_{0}-\varepsilon _{0}\right) \cup B(x_{0},\frac{3}{4}r)\right\} $,%
\begin{equation*}
\Psi _{1}\left( y\right) \in B\left( K,c_{0}-\eta \right) .
\end{equation*}

Combining this with \ref{Stays in R ball} it follows that 
\begin{equation*}
\Psi _{1}\left( B\overline{\left( K,c_{0}+\frac{\cos \left( \alpha
_{0}\right) }{100}r\right) }\right) \subset B\left( K,c_{0}-\eta \right)
\cup B(x_{0},r).
\end{equation*}%
Thus concatenating $\Phi $ with $\Psi $ gives a deformation retraction of $%
\overline{B\left( K,c_{0}+\varepsilon _{0}\right) }$ onto a subset of $%
B\left( K,c_{0}-\eta \right) \cup B(x_{0},r).$
\end{proof}

\section{Useful Flows on $\mathbb{R}^{n}\label{flow section}$}

In this section, we study certain flows on $\mathbb{R}^{n}$ and $%
S^{n-1}\subset \mathbb{R}^{n}$ that in the next section, are transferred to $%
M$ via normal coordinates to prove Theorem \ref{mainthm}. The main results
are Proposition \ref{Euclidean angle isotopy prop} and Theorem \ref{solu
flow thm}, which are used in the proofs of Case 2 and Case 3 of Theorem \ref%
{mainthm} respectively.

\begin{proposition}
\label{Euclidean angle isotopy prop}Write $\mathbb{R}^{n}=\mathbb{R}%
^{p+1}\oplus \mathbb{R}^{q+1}$ and let $S^{p}\subset \mathbb{R}^{p}$ and $%
S^{q}\subset \mathbb{R}^{q}$ be the unit spheres in $\mathbb{R}^{p+1}$ and $%
\mathbb{R}^{q+1}$ respectively. Then on the unit sphere $S^{n-1}\subset $ $%
\mathbb{R}^{n}:$

\noindent 1. $\emph{dist}\left( S^{p},\cdot \right) :S^{n-1}\setminus
\left\{ S^{p}\cup S^{q}\right\} \longrightarrow \mathbb{R}$ is smooth and
has no critical points$.$

\noindent 2. Let $\Uparrow $ be any closed subset of $S^{p}$ so that $%
S^{p}\subset B\left( \Uparrow ,\alpha \right) $ for some $\alpha \in \left(
0,\frac{\pi }{2}\right) .$ Then \emph{dist}$\left( \Uparrow ,\cdot \right)
:S^{n-1}\longrightarrow \mathbb{R}$ has no critical points on $%
S^{n-1}\setminus \left\{ S^{p}\cup S^{q}\right\} .$ In fact, $\mathrm{\func{%
grad}}\left( \emph{dist}\left( S^{p},\cdot \right) \right) $ is
gradient-like for \emph{dist}$\left( \Uparrow ,\cdot \right) $ on $%
S^{n-1}\setminus \left\{ S^{p}\cup S^{q}\right\} .$
\end{proposition}

\begin{proof}
Every point $P\in $ $S^{n-1}\setminus \left\{ S^{p}\cup S^{q}\right\} $ can
be written uniquely as 
\begin{equation*}
P=\left( X\sin \theta ,Y\cos \theta \right) ,\text{ }
\end{equation*}%
where $X\in S^{p},$ $Y\in S^{q},$ and $\theta \in \left( 0,\frac{\pi }{2}%
\right) .$ Then \textrm{dist}$\left( S^{p},\cdot \right) =\theta $ and hence
is smooth on $S^{n-1}\setminus \left\{ S^{p}\cup S^{q}\right\} $. Its
gradient is 
\begin{equation*}
\left. \mathrm{\func{grad}}\left( \mathrm{dist}\left( S^{p},\cdot \right)
\right) \right\vert _{P}=\left( X\cos \theta ,-Y\sin \theta \right)
=-\uparrow _{P}^{X};
\end{equation*}%
so Part $1$ is proven.

For Part $2,$ we start with $P=\left( X\sin \theta ,Y\cos \theta \right) \in
S^{n-1}\setminus \left\{ S^{p}\cup S^{q}\right\} ,$ and let $\Gamma $ be any
member of $\Uparrow $ with 
\begin{equation*}
\mathrm{dist}\left( \Gamma ,X\right) =\mathrm{dist}\left( \Uparrow ,X\right)
<\alpha <\frac{\pi }{2}.
\end{equation*}%
Applying the Law of Spherical Cosines to the right triangle $\Delta \left(
\Gamma ,X,P\right) $ gives%
\begin{equation}
\cos \left( \mathrm{dist}\left( \Gamma ,P\right) \right) =\cos \left( 
\mathrm{dist}\left( \Gamma ,X\right) \right) \cos \left( \mathrm{dist}\left(
X,P\right) \right) .  \label{right triangle equa}
\end{equation}%
Since 
\begin{equation}
\text{ }0<\mathrm{dist}\left( \Gamma ,X\right) <\frac{\pi }{2}\text{ \hspace{%
0.1in}and\hspace{0.1in} }0<\mathrm{dist}\left( X,P\right) <\frac{\pi }{2},
\label{accute dist
inequal}
\end{equation}%
the right-hand side of (\ref{right triangle equa}) is positive. It follows
that%
\begin{equation}
0<\cos \left( \mathrm{dist}\left( \Gamma ,P\right) \right) <\cos \left( 
\mathrm{dist}\left( \Gamma ,X\right) \right) .  \label{cos comp Inequal}
\end{equation}

A further consequence of Equation \ref{right triangle equa} is that for $%
\Gamma \in \Uparrow ,$ 
\begin{equation*}
\mathrm{dist}\left( \Gamma ,X\right) =\mathrm{dist}\left( \Uparrow ,X\right) 
\text{ if and only if }\mathrm{dist}\left( \Gamma ,P\right) =\mathrm{dist}%
\left( \Uparrow ,P\right) .
\end{equation*}%
Applying the Law of Spherical Cosines to the triangle $\Delta \left( \Gamma
,P,X\right) $ yields%
\begin{equation*}
\cos \left( \sphericalangle \left( X,P,\Gamma \right) \right) =\frac{\cos
\left( \mathrm{dist}\left( X,\Gamma \right) \right) -\cos \left( \mathrm{dist%
}\left( \Gamma ,P\right) \right) \cos \left( \mathrm{dist}\left( P,X\right)
\right) }{\sin \left( \mathrm{dist}\left( \Gamma ,P\right) \right) \sin
\left( \mathrm{dist}\left( P,X\right) \right) }.
\end{equation*}%
Combined with Inequality \ref{cos comp Inequal}, this gives $\sphericalangle
\left( X,P,\Gamma \right) \in \left( 0,\frac{\pi }{2}\right) .$ Since $%
\Gamma $ was chosen to be any member of $\Uparrow $ with $\mathrm{dist}%
\left( \Gamma ,P\right) =\mathrm{dist}\left( \Uparrow ,P\right) $, it
follows that 
\begin{equation*}
-\uparrow _{P}^{X}=\mathrm{\func{grad}}\left( \mathrm{dist}\left(
S^{p},\cdot \right) \right) |_{P}
\end{equation*}%
is gradient-like for \textrm{dist}$\left( \Uparrow ,\cdot \right) $ at $P.$
Since $P$ was an arbitrary point in $S^{n-1}\setminus \left\{ S^{p}\cup
S^{q}\right\} ,$ Part 2 follows.
\end{proof}

For the remainder of the section, we let $\Uparrow $ be a $\frac{\pi }{2}$%
--net in $S^{n-1}$ for which $\partial A(\Uparrow )\neq \emptyset .$ By
definition, $A(\Uparrow )$ consists of unit tangent vectors at least $\frac{{%
\pi }}{{2}}$ away from $\Uparrow $. So $A(\Uparrow )$ is the intersection of 
$\frac{\pi }{2}$--balls in the unit sphere $S^{n-1}.$ Together with the
hypothesis $\partial A(\Uparrow )\neq \emptyset ,$ it follows that there is
a vector $w_{s}$ in $A(\Uparrow )$ such that 
\begin{equation}
A(\Uparrow )\subset \overline{B\left( {w_{s},\frac{{\pi }}{{2}}}\right) }.
\label{ballaroundsoul}
\end{equation}%
After applying a linear isometry of $\mathbb{R}^{n},$ we may assume that ${%
w_{s}=e}_{1}\equiv \left( 1,0,\ldots ,0\right) \in \mathbb{R}^{n}.$

The rest of this section is devoted to proving the following theorem, which,
apart from Proposition \ref{Euclidean angle isotopy prop}, is the only
result of this section that is directly used in the remainder of the paper.

\begin{theorem}
\label{solu flow thm}Let $\Uparrow $ be a $\frac{\pi }{2}$--net in $S^{n-1}$
for which $\partial A(\Uparrow )\neq \emptyset $ and 
\begin{equation*}
A(\Uparrow )\subset \overline{B\left( {e}_{1}{,\frac{{\pi }}{{2}}}\right) }.
\end{equation*}

There is a constant $\alpha _{0}\in \left( 0,\frac{\pi }{2}\right) $ so that
given $R>0,$ there is a flow $\Omega _{t}$ of $\mathbb{R}^{n}$ so that for
all $t\in \left[ 0,1\right] $ and all $y\in B\left( 0,R\right) ,$%
\begin{eqnarray}
\Omega _{t}\left( y\right) &\subset &B\left( 0,\left\vert y\right\vert +%
\frac{R}{\sqrt{10}}\right) ,  \label{near origin Inequal} \\
\Omega _{1}\left( B\left( 0,R\right) \right) &\subset &\mathbb{R}%
^{n}\setminus B\left( 0,\frac{R}{\sqrt{10}}\right) ,
\label{away fr org act
inequal}
\end{eqnarray}%
\begin{equation}
\sphericalangle \left( \Omega _{1}(B\left( 0,R\right) ),\Uparrow \right)
\leq \alpha _{0}<\frac{\pi }{2},  \label{Less than alpha_2}
\end{equation}%
and $\Omega _{t}$ is the identity on 
\begin{equation*}
\left\{ \mathbb{R}^{n}\setminus B\left( 0,2R\right) \right\} .
\end{equation*}
\end{theorem}

We construct $\Omega $ by modifying the flow 
\begin{equation*}
\Psi _{t}:V\mapsto V-te_{1}
\end{equation*}%
generated by $-e_{1}.$ Before doing this, we establish two preliminary
lemmas concerning $\Psi .$

\begin{lemma}
\label{Accute Lemma}\noindent 1. If $\sphericalangle \left( e_{1}{,y}\right)
\leq \frac{\pi }{2},$ then 
\begin{equation*}
\left. \frac{d}{dt}\left\vert \Psi (y,t)\right\vert \right\vert _{t=0}\leq 0.%
\text{ }
\end{equation*}%
If $\sphericalangle \left( e_{1}{,y}\right) \geq \frac{\pi }{2},$ then%
\begin{equation*}
0\leq \left. \frac{d}{dt}\left\vert \Psi (y,t)\right\vert \right\vert
_{t=0}\leq 1.
\end{equation*}

\noindent 2. Let $y^{\perp }$ be the component of $y$ that is perpendicular
to $e_{1}.$ Then for $y\neq 0,$ 
\begin{equation*}
\left. \frac{d}{dt}\cos \sphericalangle \left( \Psi (y,t),e_{1}\right)
\right\vert _{t=0}=-\frac{\left\vert y^{\perp }\right\vert ^{2}}{\left\vert
y\right\vert ^{3}}.
\end{equation*}%
In particular, for $y\notin \mathrm{span}\left\{ e_{1}\right\} ,$ 
\begin{equation*}
t\mapsto \sphericalangle \left( \Psi (y,t),e_{1}\right)
\end{equation*}%
is a strictly increasing function.

\noindent 3. For $y=e_{1},$ 
\begin{equation*}
\sphericalangle \left( \Psi (y,t),e_{1}\right) =\left\{ 
\begin{array}{cc}
0 & \text{for }t\in \left( 0,1\right) \\ 
\pi & \text{for }t>1.%
\end{array}%
\right.
\end{equation*}

\noindent 4. Set 
\begin{equation*}
t_{y}\equiv \left\{ 
\begin{array}{ll}
0 & \text{if }e_{1}\cdot y\leq 0 \\ 
e_{1}\cdot y & \text{if }e_{1}\cdot y\geq 0.%
\end{array}%
\right.
\end{equation*}%
If $e_{1}\cdot y\geq 0,$ then 
\begin{equation*}
\Psi \left( y,t_{y}\right) =y^{\perp }.
\end{equation*}
\end{lemma}

\begin{proof}
Differentiation gives 
\begin{eqnarray*}
\frac{d}{dt}\left\vert y-te_{1}\right\vert &=&\frac{d}{dt}\left[ \left(
y-te_{1}\right) \cdot \left( y-te_{1}\right) \right] ^{1/2} \\
&=&\frac{1}{2}\left[ \left( y-te_{1}\right) \cdot \left( y-te_{1}\right) %
\right] ^{-1/2}\left( 2t-2y\cdot e_{1}\right) \\
&=&\frac{t-y\cdot e_{1}}{\left\vert y-te_{1}\right\vert }.
\end{eqnarray*}%
Evaluating at $t=0,$ we find%
\begin{equation}
\left. \frac{d}{dt}\left\vert y-te_{1}\right\vert \right\vert _{t=0}=-\frac{%
y\cdot e_{1}}{\left\vert y\right\vert },  \label{norm change 3}
\end{equation}%
proving Part 1.

Since 
\begin{equation*}
\cos \sphericalangle \left( e_{1},y-te_{1}\right) =\frac{y\cdot e_{1}-t}{%
\left\vert y-te_{1}\right\vert },
\end{equation*}%
\begin{equation}
\frac{d}{dt}\cos \sphericalangle \left( e_{1},y-te_{1}\right) =-\frac{y\cdot
e_{1}}{\left\vert y-te_{1}\right\vert ^{2}}\frac{d}{dt}\left\vert
y-te_{1}\right\vert -\frac{\left\vert y-te_{1}\right\vert -t\frac{d}{dt}%
\left\vert y-te_{1}\right\vert }{\left\vert y-te_{1}\right\vert ^{2}}.
\label{angle with
e_1 eqn}
\end{equation}%
Evaluating \ref{angle with e_1 eqn} at $t=0$ and using \ref{norm change 3}
gives 
\begin{eqnarray*}
\left. \frac{d}{dt}\cos \sphericalangle \left( e_{1},y-te_{1}\right)
\right\vert _{t=0} &=&\frac{y\cdot e_{1}}{\left\vert y\right\vert ^{2}}\cdot 
\frac{y\cdot e_{1}}{\left\vert y\right\vert }-\frac{\left\vert y\right\vert 
}{\left\vert y\right\vert ^{2}} \\
&=&\frac{\left( e_{1}\cdot y\right) ^{2}-\left\vert y\right\vert ^{2}}{%
\left\vert y\right\vert ^{3}} \\
&=&-\frac{\left\vert y^{\perp }\right\vert ^{2}}{\left\vert y\right\vert ^{3}%
},
\end{eqnarray*}%
proving Part 2.

Part 3 is an immediate consequence of the definition of $\Psi .${\Huge \ }

To establish Part 4, note that if $e_{1}\cdot y\geq 0,$ we have%
\begin{eqnarray}
e_{1}\cdot \Psi \left( y,t_{y}\right) &=&e_{1}\cdot y-t_{y}\left( e_{1}\cdot
e_{1}\right)  \notag \\
&=&e_{1}\cdot y-t_{y}  \label{eqn perrrrpp} \\
&=&0,  \notag
\end{eqnarray}%
since $t_{y}=e_{1}\cdot y.$ Letting the superscript $^{\perp }$ denote the
component of a vector perpendicular to $e_{1},$ we then have 
\begin{eqnarray*}
\Psi \left( y,t_{y}\right) &=&\left( \Psi \left( y,t_{y}\right) \right)
^{\perp }\text{, by Equation }\ref{eqn perrrrpp} \\
&=&y-t_{y}e_{1} \\
&=&y^{\perp },
\end{eqnarray*}%
proving Part 4.
\end{proof}

\begin{lemma}
\label{arrive lemma}Set 
\begin{equation*}
t_{y}\equiv \left\{ 
\begin{array}{ll}
0 & \text{if }e_{1}\cdot y\leq 0 \\ 
e_{1}\cdot y & \text{if }e_{1}\cdot y\geq 0.%
\end{array}%
\right.
\end{equation*}%
For all $y\in B\left( 0,R\right) $ and all $t\in \left[ 0,\frac{R}{\sqrt{10}}%
\right] ,$ 
\begin{equation*}
\cos \sphericalangle \left( \Psi _{t_{y}+\frac{R}{\sqrt{10}}}\left( y\right)
,e_{1}\right) \leq -\sqrt{\frac{1}{11}},
\end{equation*}%
\begin{equation*}
\left\vert \Psi _{t_{y}+t}\left( y\right) \right\vert \leq \left\vert
y\right\vert +\frac{R}{\sqrt{10}}
\end{equation*}%
and 
\begin{equation*}
\frac{R}{\sqrt{10}}\leq \left\vert \Psi _{t_{y}+\frac{R}{\sqrt{10}}}\left(
y\right) \right\vert .
\end{equation*}
\end{lemma}

\begin{proof}
First we prove the inequalities when $e_{1}\cdot y\geq 0.$

Part 4 of Lemma \ref{Accute Lemma} gives that if $e_{1}\cdot y\geq 0,$ then $%
y-t_{y}e_{1}=y^{\perp },$ the component of $y$ that is perpendicular to $%
e_{1}.$ So for $t\in \left[ 0,\frac{R}{\sqrt{10}}\right] $ and $y\in B\left(
0,R\right) ,$ 
\begin{eqnarray}
\left\vert y-\left( t_{y}+t\right) e_{1}\right\vert ^{2} &=&\left\vert
y^{\perp }-te_{1}\right\vert ^{2}  \notag \\
&=&\left\vert y^{\perp }\right\vert ^{2}+\left\vert t\right\vert ^{2}  \notag
\\
&\leq &\left\vert y\right\vert ^{2}+\left\vert t\right\vert ^{2}  \notag \\
&\leq &\left\vert y\right\vert ^{2}+\frac{R^{2}}{10},\text{ so}  \notag \\
\left\vert y-\left( t_{y}+t\right) e_{1}\right\vert &\leq &\sqrt{\left\vert
y\right\vert ^{2}+\frac{R^{2}}{10}}  \label{inter Inequal} \\
&\leq &\left\vert y\right\vert +\frac{R}{\sqrt{10}}  \notag
\end{eqnarray}

so the second inequality follows when $e_{1}\cdot y\geq 0.$

To prove the first inequality when $e_{1}\cdot y\geq 0,$ note 
\begin{eqnarray*}
\cos \sphericalangle \left( e_{1},y-\left( t_{y}+t\right) e_{1}\right)
&=&e_{1}\cdot \left( \frac{y^{\perp }-te_{1}}{\left\vert y-\left(
t_{y}+t\right) e_{1}\right\vert }\right) \\
&=&-\left( \frac{t}{\left\vert y-\left( t_{y}+t\right) e_{1}\right\vert }%
\right) \\
&\leq &\left( \frac{-t}{\sqrt{R^{2}+t^{2}}}\right) ,
\end{eqnarray*}%
by Inequality \ref{inter Inequal}. Hence 
\begin{equation*}
\cos ^{2}\sphericalangle \left( e_{1},y-\left( t_{y}+t\right) e_{1}\right)
\geq \frac{t^{2}}{R^{2}+t^{2}}.
\end{equation*}%
Thus for $t=\frac{R}{\sqrt{10}},$%
\begin{eqnarray*}
\cos ^{2}\sphericalangle \left( e_{1},y-\left( t_{y}+\frac{R}{\sqrt{10}}%
\right) e_{1}\right) &\geq &\frac{\frac{R^{2}}{10}}{R^{2}+\frac{R^{2}}{10}}
\\
&=&\frac{\frac{R^{2}}{10}}{\frac{11R^{2}}{10}} \\
&=&\frac{1}{11},
\end{eqnarray*}%
and 
\begin{equation}
\cos \sphericalangle \left( e_{1},y-\left( t_{y}+\frac{R}{\sqrt{10}}\right)
e_{1}\right) \leq -\sqrt{\frac{1}{11}},  \label{angle with e_1 Inequal}
\end{equation}%
proving the first inequality when $e_{1}\cdot y\geq 0.$

To prove the third inequality when $e_{1}\cdot y\geq 0,$ note that 
\begin{eqnarray}
\left\vert \Psi _{t_{y}+\frac{R}{\sqrt{10}}}\left( y\right) \right\vert ^{2}
&=&\left\vert y^{\perp }-\frac{R}{\sqrt{10}}e_{1}\right\vert ^{2}  \notag \\
&=&\left\vert y^{\perp }\right\vert ^{2}+\frac{R^{2}}{10}  \notag \\
&\geq &\frac{R^{2}}{10},\text{ so}  \notag \\
\left\vert \Psi _{t_{y}+\frac{R}{\sqrt{10}}}\left( y\right) \right\vert
&\geq &\frac{R}{\sqrt{10}}.  \label{norm bnded below}
\end{eqnarray}

To prove the inequalities when $e_{1}\cdot y\leq 0$, we use that $y^{\perp }$
is on the backward flow line of $\Psi $ that passes through $y.$ Combining
this with the fact that 
\begin{equation*}
t\mapsto \sphericalangle \left( \Psi (y,t),e_{1}\right)
\end{equation*}%
is an increasing function gives 
\begin{eqnarray*}
\cos \sphericalangle \left( \Psi _{t_{y}+\frac{R}{\sqrt{10}}}\left( y\right)
,e_{1}\right) &=&\cos \sphericalangle \left( \Psi _{\frac{R}{\sqrt{10}}%
}\left( y\right) ,e_{1}\right) \\
&\leq &\cos \sphericalangle \left( \Psi _{\frac{R}{\sqrt{10}}}\left(
y^{\perp }\right) ,e_{1}\right) \\
&\leq &-\sqrt{\frac{1}{11}}.
\end{eqnarray*}

Combining Part 1 of Lemma \ref{Accute Lemma} with $e_{1}\cdot y\leq 0$ and
the fact that $y^{\perp }$ is on the backward flow line of $\Psi $ that
passes through $y,$ we have 
\begin{eqnarray*}
\left\vert \Psi _{t_{y}+\frac{R}{\sqrt{10}}}\left( y\right) \right\vert
&=&\left\vert \Psi _{\frac{R}{\sqrt{10}}}\left( y\right) \right\vert \\
&\geq &\left\vert \Psi _{\frac{R}{\sqrt{10}}}\left( y^{\perp }\right)
\right\vert \\
&\geq &\frac{R}{\sqrt{10}},
\end{eqnarray*}%
by Inequality \ref{norm bnded below}, and 
\begin{eqnarray*}
\left\vert \Psi _{t_{y}+\frac{R}{\sqrt{10}}}\left( y\right) \right\vert
&=&\left\vert \Psi _{\frac{R}{\sqrt{10}}}\left( y\right) \right\vert \\
&=&\left\vert y-\frac{R}{\sqrt{10}}e_{1}\right\vert \\
&\leq &\left\vert y\right\vert +\frac{R}{\sqrt{10}},
\end{eqnarray*}%
as claimed.
\end{proof}

\begin{proof}[Proof of Theorem \protect\ref{solu flow thm}]
Set 
\begin{equation*}
K_{0}\equiv \left\{ \left. z\in S^{n-1}\text{ }\right\vert \text{ }\cos
\sphericalangle \left( z,e_{1}\right) \leq -\sqrt{\frac{1}{11}}\right\} .
\end{equation*}

Our hypothesis 
\begin{equation*}
A(\Uparrow )\subset \overline{B\left( {e}_{1}{,\frac{{\pi }}{{2}}}\right) }
\end{equation*}%
implies that there is an $\alpha _{0}\in \left( 0,\frac{\pi }{2}\right) $ so
that 
\begin{equation}
\sphericalangle \left( z,\Uparrow \right) <\alpha _{0}<\frac{{\pi }}{{2}}
\label{less than alpha_0 inequal}
\end{equation}%
for all $z\in K_{0}.$

Given $R>0,$ set%
\begin{equation*}
Z=\left\{ \mathbb{R}^{n}\setminus B\left( 0,2R\right) \right\}
\end{equation*}%
and 
\begin{equation*}
O=\overline{B\left( 0,\frac{3}{2}R\right) }.
\end{equation*}

Let $f:\mathbb{R}^{n}\longrightarrow \left[ 0,1\right] $ be $C^{\infty }$
and satisfy 
\begin{equation}
f|_{Z}\equiv 0\text{ and }f|_{O}\equiv 1.  \label{dfn of f eqn}
\end{equation}%
Set 
\begin{equation*}
X\equiv -f\cdot e_{1}
\end{equation*}%
and let $\overline{\Psi }$ be the flow of $X.$

By Lemma \ref{arrive lemma}, for all $y\in B\left( 0,R\right) $ and all $%
t\in \left[ 0,\frac{R}{\sqrt{10}}\right] ,$ 
\begin{eqnarray}
\cos \sphericalangle \left( \Psi _{t_{y}+\frac{R}{\sqrt{10}}}\left( y\right)
,e_{1}\right) &\leq &-\sqrt{\frac{1}{11}},\text{ }\left\vert \Psi
_{t_{y}+t}\left( y\right) \right\vert \leq \left\vert y\right\vert +\frac{R}{%
\sqrt{10}},\text{ and}  \label{psi both inequal} \\
\frac{R}{\sqrt{10}} &\leq &\left\vert \Psi _{t_{y}+\frac{R}{\sqrt{10}}%
}\left( y\right) \right\vert .  \notag
\end{eqnarray}%
Combining this with Inequality \ref{less than alpha_0 inequal}, we see that
for all $y\in B\left( 0,R\right) ,$%
\begin{equation}
\sphericalangle \left( \Psi _{t_{y}+\frac{R}{\sqrt{10}}}\left( y\right)
,\Uparrow \right) <\alpha _{0}.  \label{Psi angle Inequal}
\end{equation}

The second inequality in \ref{psi both inequal} combined with the definition
of $\overline{\Psi }$ gives that for all $y\in B\left( 0,R\right) $ and all $%
t\in \left[ 0,\frac{R}{\sqrt{10}}\right] ,$ 
\begin{equation*}
\overline{\Psi }_{t_{y}+t}\left( y\right) =\Psi _{t_{y}+t}\left( y\right) .
\end{equation*}%
So Inequalities \ref{psi both inequal} and \ref{Psi angle Inequal} give that
for all $y\in B\left( 0,R\right) $ and all $t\in \left[ 0,\frac{R}{\sqrt{10}}%
\right] ,$ 
\begin{eqnarray*}
\sphericalangle \left( \overline{\Psi }_{t_{y}+\frac{R}{\sqrt{10}}}\left(
y\right) ,\Uparrow \right) &<&\alpha _{0},\text{ }\left\vert \overline{\Psi }%
_{t_{y}+t}\left( y\right) \right\vert \leq \left\vert y\right\vert +\frac{R}{%
\sqrt{10}},\text{ and} \\
\frac{R}{\sqrt{10}} &\leq &\left\vert \overline{\Psi }_{t_{y}+\frac{R}{\sqrt{%
10}}}\left( y\right) \right\vert .
\end{eqnarray*}

So we define $\Omega $ to be 
\begin{equation*}
\Omega :\left( y,t\right) \mapsto \overline{\Psi }_{\left( t_{y}+\sqrt{\frac{%
R}{10}}\right) t}\left( y\right) ,
\end{equation*}%
and note that the previous three displayed lines combined with Part 1 of
Lemma \ref{Accute Lemma} give us Inequalities \ref{near origin Inequal}, \ref%
{away fr org act inequal}, and \ref{Less than alpha_2}, and Equation \ref%
{dfn of f eqn} gives that $\Omega _{t}$ is the identity on 
\begin{equation*}
\left\{ \mathbb{R}^{n}\setminus B\left( 0,2R\right) \right\} .
\end{equation*}
\end{proof}

\section{Proof of the Connectedness Theorem}

\label{proofofresults}

In this section, we prove Theorem \ref{mainthm}. Recall from Section \ref%
{background section} that there is an $\varepsilon _{0}>0$ such that $x_{0}$
is the only critical point for \textrm{dist}$_{K}$ in $\overline{B\left(
K,c_{0}+\varepsilon _{0}\right) }\setminus B\left( K,c_{0}-\varepsilon
_{0}\right) .$

The proof of Theorem \ref{mainthm} is divided into three cases corresponding
to the three cases in the definition of sub-index.


\textbf{Case 1:} Suppose $A(\Uparrow _{x_{0}}^{K})=\emptyset $. Then for all 
$v\in S_{x_{0}}$, we have $\sphericalangle (v,\Uparrow _{x_{0}}^{K})<\frac{%
\pi }{2}$, and by compactness of $S_{x_{0}},$ there is an $\alpha $ so that 
\begin{equation*}
\sphericalangle (v,\Uparrow _{x_{0}}^{K})<\alpha <\frac{\pi }{2}
\end{equation*}%
for all $v\in S_{x_{0}}.$ Combining this with Lemma \ref{Taylorpolylemma},
it follows that 
\begin{equation*}
\mathrm{dist}_{K}\left( {exp_{x_{0}}\left( tv\right) }\right) \leq c_{0}-%
\frac{t}{2}\cdot \cos \left( \alpha \right)
\end{equation*}%
for all $v\in S_{x_{0}}$ and all sufficiently small $t.$ In particular, the
distance between $x_{0}$ and $K$ decreases regardless of the direction we
travel away from $x_{0}$. Thus the point $x_{0}$ is a strict local maximum
for $\mathrm{dist}_{K}$.

Suppose $\iota :E^{k}\longrightarrow B\left( K,c_{0}+\varepsilon _{0}\right) 
$ is a cell with $\mathrm{\dim }\left( E^{k}\right) =k\leq n-1$ and $\iota
\left( \partial E^{k}\right) \in B\left( K,c_{0}-\varepsilon _{0}\right) .$
After applying Lemma \ref{localreduction}, we may assume that $\iota \left(
E^{k}\right) \subset $ $B\left( K,c_{0}-\eta \right) \cup B(x_{0},R)$ for
any sufficiently small $R>0.$

Use transversality (see, e.g., Theorem 14.7 of \cite{BrockJan}) to deform $%
\iota $ so that $x_{0}$ is not in its image. It follows that $\iota \left(
E^{k}\right) \subset $ $B\left( K,c_{0}\right) .$ Since there are no
critical points for $\mathrm{dist}_{K}$ on $B\left( K,c_{0}\right) \setminus
B\left( K,c_{0}-\varepsilon _{0}\right) ,$ it follows that we can further
deform $\iota $ into $B\left( K,c_{0}-\varepsilon _{0}\right) .$

\textbf{Setup for Case 2 (and mostly also for Case 3): }Here we describe our
setup for Case 2. With a few modifications, it will also be our setup for
Case 3. Suppose $\iota :E^{k}\longrightarrow B\left( K,c_{0}+\varepsilon
_{0}\right) $ is a cell with $\mathrm{\dim }\left( E^{k}\right) =k$ and $%
\iota \left( \partial E^{k}\right) \in B\left( K,c_{0}-\varepsilon
_{0}\right) .$ As before, we construct a homotopy of $\iota $ into $B\left(
K,c_{0}-\varepsilon _{0}\right) $ that fixes $\iota |_{\partial E^{k}}.$
Since there are no critical points for $\mathrm{dist}_{K}$ on $B\left(
K,c_{0}\right) \setminus B\left( K,c_{0}-\varepsilon _{0}\right) ,$ it is
sufficient to construct a homotopy of $\iota $ to a cell whose image is in $%
B\left( K,c_{0}\right) .$

Let $S^{p}$ be the unit sphere in $\mathrm{span}\left( {\Uparrow _{x_{0}}^{K}%
}\right) ,$ and set 
\begin{equation*}
\bar{U}_{\frac{\pi }{4}}=\left\{ \left. v\in S_{x_{0}}\text{ }\right\vert 
\text{ }\sphericalangle \left( v,S^{p}\right) \leq \frac{\pi }{4}\right\} .
\end{equation*}%
There is an $\tilde{\alpha}_{0}\in \left( 0,\frac{\pi }{2}\right) $ so that
for all $v\in \bar{U}_{\frac{\pi }{4}},$ 
\begin{equation}
\sphericalangle \left( v,{\Uparrow _{x_{0}}^{K}}\right) <\tilde{\alpha}_{0}<%
\frac{\pi }{2}.  \label{dfn of alpha_0}
\end{equation}

It follows from Lemma \ref{Taylorpolylemma} that there is an $R_{1}>0$ so
that for all $\rho \in \left[ 0,R_{1}\right] $ and all $y\in B\left(
x_{0},2\rho \right) \setminus $ $B\left( x_{0},\frac{\rho }{2}\right) $ with 
$\sphericalangle \left( \Uparrow _{x_{0}}^{y},\Uparrow _{x_{0}}^{K}\right)
\in \left( 0,\tilde{\alpha}_{0}\right) ,$

\begin{equation}
\mathrm{dist}_{K}\left( y\right) <c_{0}-\frac{\rho }{3}\cos \tilde{\alpha}%
_{0}.  \label{bleow c_0 Ineqaul}
\end{equation}

Choose $4r<\min \left\{ \frac{1}{2}\mathrm{inj}_{x_{0}},\text{ }\varepsilon
_{0},R_{1}\right\} .$ Then $\overline{B(x_{0},2r)}$ is contained in $B\left(
K,c_{0}+\varepsilon _{0}\right) \setminus \overline{B\left(
K,c_{0}+\varepsilon _{0}\right) }.$ So after applying Lemma \ref%
{localreduction}, we may assume 
\begin{equation*}
\iota \left( E^{k}\right) \subset B\left( K,c_{0}-\eta \right) \cup
B(x_{0},r).
\end{equation*}%
In particular, 
\begin{equation}
\iota \left( E^{k}\right) \setminus B(x_{0},r)\subset B\left( K,c_{0}-\eta
\right) .  \label{local red consq Statement}
\end{equation}

Finally, 
\begin{eqnarray*}
\mathrm{dist}_{K}B(x_{0},4r) &\geq &c_{0}-4r \\
&>&c_{0}-\varepsilon _{0},
\end{eqnarray*}%
and $\iota \left( \partial E^{k}\right) \subset B\left( K,c_{0}-\varepsilon
_{0}\right) .$ Thus $\iota \left( \partial E^{k}\right) \cap
B(x_{0},4r)=\emptyset .$

\textbf{Case 2:} Suppose $A(\Uparrow _{x_{0}}^{K})\neq \emptyset $, $%
\partial A(\Uparrow _{x_{0}}^{K})=\emptyset $, and 
\begin{equation*}
\mathrm{dim}\left( E^{k}\right) =k\leq n-\dim \mathrm{span}\left\{
A(\Uparrow _{x_{0}}^{K})\right\} -1.
\end{equation*}%
Define 
\begin{equation*}
C_{2r}A(\Uparrow _{x_{0}}^{K})\equiv \left\{ \left. \exp {_{x_{0}}\left( {%
tA(\Uparrow _{x_{0}}^{K})}\right) }\text{ }\right\vert {\ t\in \lbrack 0,2r]}%
\right\} .
\end{equation*}%
Note 
\begin{eqnarray*}
\dim E^{k}+\dim C_{2r}A(\Uparrow _{x_{0}}^{K}) &=&k+{\dim \mathrm{span}}%
\left\{ {A(\Uparrow _{x_{0}}^{K})}\right\} \\
&\leq &n-\dim \mathrm{span}\left\{ A(\Uparrow _{x_{0}}^{K})\right\} -1+{\dim 
\mathrm{span}}\left\{ {A(\Uparrow _{x_{0}}^{K})}\right\} \\
&<&n.
\end{eqnarray*}%
So by transversality (see, e.g., Theorem 14.7 of \cite{BrockJan}), there is
a homotopy $\iota _{t}$ of $\iota $ so that 
\begin{equation}
\iota _{1}\left( {E^{k}}\right) \bigcap C_{2r}A(\Uparrow
_{x_{0}}^{K})=\emptyset ,  \label{tranverslity statement}
\end{equation}%
and $\iota _{1}$ agrees with $\iota $ on $\iota _{1}^{-1}\left( M\setminus
B\left( {x_{0},3r}\right) \right) $. We may, moreover, choose $\iota _{t}$
so that for all $t,$ $\iota _{t}$ is arbitrarily close to $\iota .$ Abusing
notation we call $\iota _{1},$ $\iota .$ The last ingredient in the proof of
Case 2 is the following lemma.

\begin{lemma}
\label{case 2 Lemma}Set 
\begin{equation*}
\tilde{\eta}\equiv \mathrm{min}\left\{ \eta ,\frac{r}{2}\cos \tilde{\alpha}%
_{0}\right\} .
\end{equation*}%
There is a deformation retraction $H_{t}$ of $M\setminus C_{2r}A(\Uparrow
_{x_{0}}^{K})$ with the following properties:

\noindent 1. $H_{t}$ fixes both $M\setminus B(x_{0},2r)$ and $B\left(
K,c_{0}-\tilde{\eta}\right) \setminus B\left( {x_{0},\frac{{1}}{{2}}r}%
\right) .$

\noindent 2. $H_{t}$ restricts to a strong deformation retract of $%
B(x_{0},r)\setminus C_{r}A(\Uparrow _{x_{0}}^{K})$ to a subset of $B\left(
K,c_{0}-\frac{\tilde{\eta}}{2}\right) \setminus B\left( {x_{0},\frac{{1}}{{2}%
}r}\right) .$
\end{lemma}

\begin{proof}
We construct $H_{t}$ by concatenating two homotopies, which we call the
Radial Homotopy and the Angle Homotopy.

To construct the Radial Homotopy, use radial geodesics from $x_{0}$ to
deform $B(x_{0},2r)\setminus \{x_{0}\}$ onto 
\begin{equation*}
B(x_{0},2r)\setminus B\left( {x_{0},\frac{{1}}{{2}}r}\right) .
\end{equation*}%
Since $C_{2r}A\left( \Uparrow _{x_{0}}^{K}\right) $ is a union of radial
geodesics, this restricts to a deformation of 
\begin{equation*}
B(x_{0},2r)\setminus C_{2r}A(\Uparrow _{x_{0}}^{K})
\end{equation*}%
onto 
\begin{equation*}
B(x_{0},2r)\setminus \left\{ {B\left( {x_{0},\frac{{1}}{{2}}r}\right) \cup
C_{2r}A(\Uparrow _{x_{0}}^{K})}\right\} .
\end{equation*}

To construct the Angle Homotopy, apply Part 2 of Proposition \ref{Euclidean
angle isotopy prop} with $S^{p}$ the unit sphere in $\mathrm{span}\left( {%
\Uparrow _{x_{0}}^{K}}\right) $ and $S^{q}$ the unit sphere in $\mathrm{span}%
\left( A\left( {\Uparrow _{x_{0}}^{K}}\right) \right) .$ This gives an
isotopy $\widehat{\mathcal{A}}_{t}$ of $S_{x_{0}}\setminus \left\{ S^{p}\cup
S^{q}\right\} $ onto a subset of 
\begin{equation*}
\bar{U}_{\frac{\pi }{4}}=\left\{ \left. v\in S_{x_{0}}\text{ }\right\vert 
\text{ }\sphericalangle \left( v,S^{p}\right) \leq \frac{\pi }{4}\right\} .
\end{equation*}%
Extending $\widehat{\mathcal{A}}_{t}$ radially gives an isotopy $\mathcal{A}%
_{t}$ of 
\begin{equation*}
\left\{ \left. v\in T_{x_{0}}M\setminus \mathrm{span}\left\{ {A(\Uparrow
_{x_{0}}^{K})}\right\} \text{ }\right\vert \text{ }{\frac{{1}}{{2}}r}\leq
\left\vert v\right\vert \leq 2r\right\}
\end{equation*}%
to a subset of 
\begin{equation*}
\left\{ \left. \text{ }v\in T_{x_{0}}M\right\vert \text{ }{\frac{{1}}{{2}}r}%
\leq \left\vert v\right\vert \leq 2r\text{ and }\frac{v}{\left\vert
v\right\vert }\in \bar{U}_{\frac{\pi }{4}}\right\} .
\end{equation*}

Pre- and post-composing $\mathcal{A}_{t}$ with $\exp _{x_{0}}$ then gives an
isotopy $\overline{\mathcal{A}}_{t}$ of ${B\left( {x_{0},2r}\right) }%
\setminus C_{2r}A(\Uparrow _{x_{0}}^{K}).$ By Inequalities \ref{dfn of
alpha_0} and \ref{bleow c_0 Ineqaul}, $\overline{\mathcal{A}}_{1}$ takes ${%
B\left( {x_{0},2r}\right) }\setminus \left\{ {B\left( {x_{0},\frac{{1}}{{2}}r%
}\right) \cup }C_{2r}A(\Uparrow _{x_{0}}^{K})\right\} $ to a subset of 
\begin{equation}
B\left( K,c_{0}-\frac{r}{3}\cos \tilde{\alpha}_{0}\right) \setminus B\left( {%
x_{0},\frac{{1}}{{2}}r}\right) \subset B\left( K,c_{0}-\frac{\tilde{\eta}}{2}%
\right) \setminus B\left( {x_{0},\frac{{1}}{{2}}r}\right) .
\label{into eta/2 in equal}
\end{equation}

Let $X$ be the vector field whose flow gives $\overline{\mathcal{A}}_{t},$
and let $\varphi :M\longrightarrow \mathbb{R}$ be $C^{\infty }$ and satisfy 
\begin{equation}
\varphi \left( x\right) =\left\{ 
\begin{array}{ll}
1 & \text{for }x\in \overline{B\left( x_{0},r\right) }\setminus B\left(
K,c_{0}-\frac{\tilde{\eta}}{2}\right) \\ 
0 & \text{for }x\in \left\{ M\setminus B\left( x_{0},2r\right) \right\} \cup
\left\{ B\left( K,c_{0}-\tilde{\eta}\right) \setminus B\left( x_{0},\frac{r}{%
2}\right) \right\} .%
\end{array}%
\right.  \label{dfn of phi}
\end{equation}%
Let $\widetilde{\mathcal{A}}_{t}$ be the flow generated by $\varphi X.$ Then 
$\widetilde{\mathcal{A}}_{t}$ fixes $\left\{ M\setminus B\left(
x_{0},2r\right) \right\} \cup \left\{ B\left( K,c_{0}-\tilde{\eta}\right)
\setminus B\left( x_{0},\frac{r}{2}\right) \right\} $.

It follows from \ref{into eta/2 in equal} and Part 2 of Proposition \ref%
{Euclidean angle isotopy prop} that $\widetilde{\mathcal{A}}_{1}$ takes ${%
B\left( {x_{0},r}\right) }\setminus \left\{ {B\left( {x_{0},\frac{{1}}{{2}}r}%
\right) \cup }C_{r}A(\Uparrow _{x_{0}}^{K})\right\} $ to a subset of $%
B\left( K,c_{0}-\frac{\tilde{\eta}}{2}\right) \setminus B\left( {x_{0},\frac{%
{1}}{{2}}r}\right) .$

Thus concatenating $\widetilde{\mathcal{A}}_{t}$ with the Radial Homotopy,
yields the desired homotopy $H_{t}.$
\end{proof}

To finish the proof of Case 2, we note that since $H_{t}$ fixes $\left\{
B\left( K,c_{0}-\tilde{\eta}\right) \setminus B\left( x_{0},\frac{r}{2}%
\right) \right\} $ and $\tilde{\eta}\leq \eta ,$ Statement \ref{local red
consq Statement} gives%
\begin{equation*}
H_{t}|_{\mathrm{image}\left( \iota \right) \cap M\setminus B(x_{0},r)}=%
\mathrm{id.}
\end{equation*}%
Combining this with \ref{local red consq Statement} we get,%
\begin{equation*}
H_{t}|_{\mathrm{image}\left( \iota \right) \cap M\setminus
B(x_{0},r)}\subset B\left( K,c_{0}-\eta \right) .
\end{equation*}%
By Statement \ref{tranverslity statement},%
\begin{equation*}
\left\{ \iota _{1}\left( {E^{k}}\right) {\cap B\left( {x_{0},2r}\right) }%
\right\} \bigcap C_{2r}A(\Uparrow _{x_{0}}^{K})=\emptyset .
\end{equation*}%
Combining this with Part 2 of Lemma \ref{case 2 Lemma} and our abuse of
notation that $\iota _{1}=\iota $, it follows that 
\begin{equation*}
H_{1}\left( \mathrm{image}\left( \iota \right) \cap B(x_{0},r)\right)
\subset B\left( K,c_{0}-\frac{\tilde{\eta}}{2}\right) .
\end{equation*}%
Therefore $H_{1}\left( \mathrm{image}\left( \iota \right) \right) \subset
B\left( K,c_{0}\right) ,$ and $H_{t}|_{\iota \left( \partial E^{k}\right)
}\equiv \mathrm{id.}$ Since there are no critical points for $\mathrm{dist}%
_{K}$ on $B\left( K,c_{0}\right) \setminus B\left( K,c_{0}-\varepsilon
_{0}\right) ,$ this completes the proof of Case 2.

\textbf{Case 3:} Suppose $A(\Uparrow _{x_{0}}^{K})\neq \emptyset $ and $%
\partial A(\Uparrow _{x_{0}}^{K})\neq \emptyset $. Let $\iota
:E^{k}\longrightarrow B\left( K,c_{0}+\varepsilon _{0}\right) $ be a cell
with $\iota \left( \partial E^{k}\right) \in B\left( K,c_{0}-\varepsilon
_{0}\right) .$ As before, we construct a homotopy of $\iota $ into $B\left(
K,c_{0}-\varepsilon _{0}\right) $ that fixes $\iota |_{\partial E^{k}}.$

Since $A(\Uparrow _{x_{0}}^{K})\neq \emptyset $ and $\partial A(\Uparrow
_{x_{0}}^{K})\neq \emptyset ,$ $\Uparrow _{x_{0}}^{K},$ viewed as a subset
of the unit tangent sphere at $x_{0}$ and after possibly applying a linear
isometry, satisfies the hypotheses of Theorem \ref{solu flow thm}. Let $%
\alpha _{0}$ be as in Theorem \ref{solu flow thm}.

By Lemma \ref{Taylorpolylemma}, there is an $R_{1}>0$ so that for all $\rho
\in \left[ 0,R_{1}\right] $ and all $y\in B\left( x_{0},4\rho \right)
\setminus $ $B\left( x_{0},\frac{\rho }{\sqrt{10}}\right) $ with $%
\sphericalangle \left( \Uparrow _{x_{0}}^{y},\Uparrow _{x_{0}}^{K}\right)
\in \left( 0,\alpha _{0}\right) ,$%
\begin{equation}
\mathrm{dist}_{K}\left( y\right) <c_{0}-\frac{\rho }{2\sqrt{10}}\cos \left(
\alpha _{0}\right) .  \label{below inequal 2}
\end{equation}

Choose $2r<\min \left\{ \frac{1}{2}\mathrm{inj}_{x_{0}},\text{ }\varepsilon
_{0},R_{1}\right\} .$ Then 
\begin{equation*}
\overline{B(x_{0},2r)}\subset B\left( K,c_{0}+\varepsilon _{0}\right)
\setminus \overline{B\left( K,c_{0}-\varepsilon _{0}\right) }.
\end{equation*}%
So after applying Lemma \ref{localreduction}$,$ we may assume that 
\begin{equation}
\iota \left( E^{k}\right) \subset B\left( K,c_{0}-\eta \right) \cup
B(x_{0},r).  \label{local reduct Ineqaul}
\end{equation}

In particular, 
\begin{equation}
\iota \left( E^{k}\right) \setminus B(x_{0},r)\subset B\left( K,c_{0}-\eta
\right) .  \label{2nd local red Ineqaul}
\end{equation}

As in Case 2, 
\begin{eqnarray*}
\mathrm{dist}_{K}B(x_{0},2r) &\geq &c_{0}-2r \\
&>&c_{0}-\varepsilon _{0},
\end{eqnarray*}%
and $\iota \left( \partial E^{k}\right) \subset B\left( K,c_{0}-\varepsilon
_{0}\right) .$ Thus 
\begin{equation}
\iota \left( \partial E^{k}\right) \cap B(x_{0},2r)=\emptyset .
\label{boundary far}
\end{equation}

Apply Theorem \ref{solu flow thm} with $R=2r$. Let $\Omega $ be the flow of $%
T_{x_{0}}M$ this produces. Let $\Phi $ be the flow of $M$ obtained by pre-
and post-composing $\Omega $ by $\exp _{x_{0}},$ and let $X$ be the vector
field that generates $\Omega .$

If $z\in \iota \left( E^{k}\right) \cap B(x_{0},2r),$ then by Inequalities %
\ref{near origin Inequal}, \ref{away fr org act inequal} and \ref{Less than
alpha_2}, we have%
\begin{equation*}
2\frac{r}{\sqrt{10}}\leq \mathrm{dist}_{x_{0}}\left( \Phi _{1}(z)\right) ,
\end{equation*}%
\begin{equation*}
\sphericalangle \left( \exp _{x_{0}}^{-1}\left( \Phi _{1}(z)\right)
,\Uparrow _{x_{0}}^{K}\right) \leq \alpha _{0}<\frac{\pi }{2},
\end{equation*}%
and%
\begin{equation}
\mathrm{dist}_{x_{0}}\left( \Phi _{t}(z)\right) \leq \mathrm{dist}%
_{x_{0}}\left( z\right) +\frac{2r}{\sqrt{10}},\text{ \label{where end up
Inequal}}
\end{equation}%
for all $t\in \left[ 0,1\right] .$ Therefore by Inequality \ref{below
inequal 2},%
\begin{equation}
\mathrm{dist}_{K}\left( \Phi _{1}(z)\right) <c_{0}-\frac{r}{2\sqrt{10}}\cos
\left( \alpha _{0}\right) .  \label{PHi_1 good Inequal}
\end{equation}%
Set $\tilde{\eta}=\mathrm{min}\left\{ \eta ,\frac{r}{2\sqrt{10}}\cos \left(
\alpha _{0}\right) \right\} .$ It follows that for all $z\in B(x_{0},2r),$
there is a first time $t_{z}$ that varies continuously with $z$ so that

\begin{equation}
\Phi _{t_{z}}(z)\in \overline{B\left( K,c_{0}-\frac{\tilde{\eta}}{2}\right) }%
.  \label{Phi arrives Statement}
\end{equation}

Let $\psi :M\longrightarrow \mathbb{R}$ be $C^{\infty }$ and satisfy 
\begin{equation*}
\psi =\left\{ 
\begin{array}{ll}
1 & B\left( x_{0},\frac{7}{4}r\right) \\ 
0 & M\setminus B\left( x_{0},\frac{9}{5}r\right) .%
\end{array}%
\right.
\end{equation*}%
Let $\tilde{\Phi}$ be the flow of $\psi X,$ and set 
\begin{equation*}
\Upsilon _{t}\left( z\right) =\tilde{\Phi}_{t_{z}\cdot t}(z).
\end{equation*}%
Combining the definition of $\Upsilon $ with Inequality \ref{where end up
Inequal}, we have 
\begin{equation*}
\Upsilon _{t}\left( z\right) =\Phi _{t_{z}\cdot t}(z),
\end{equation*}%
for all $t\in \left[ 0,1\right] $ and all $z\in B(x_{0},r).$ Therefore by %
\ref{Phi arrives Statement}, we have%
\begin{equation*}
\Upsilon _{1}(z)\in \overline{B\left( K,c_{0}-\frac{\tilde{\eta}}{2}\right) }
\end{equation*}%
for all $z\in B(x_{0},r).$

For $z\in \iota \left( E^{k}\right) \setminus B(x_{0},r)$, \ref{2nd local
red Ineqaul} gives 
\begin{equation*}
z\in B\left( K,c_{0}-\eta \right) \subset \overline{B\left( K,c_{0}-\frac{%
\tilde{\eta}}{2}\right) }.
\end{equation*}%
So $t_{z}=0,$ and 
\begin{equation*}
\Upsilon _{t}\left( z\right) =\Phi _{t_{z}\cdot t}(z)=z
\end{equation*}%
for all $t.$

Thus $\iota :E^{k}\longrightarrow M$ is homotopic to a cell in $B\left(
K,c_{0}\right) $ via a homotopy that fixes $\partial E^{k}.$ Since there are
no critical points for $\mathrm{dist}_{K}$ on $B\left( K,c_{0}\right)
\setminus B\left( K,c_{0}-\varepsilon _{0}\right) ,$ it follows that we can
further deform $\iota $ into $B\left( K,c_{0}-\varepsilon _{0}\right) .$ $%
\square $

\section{Lemmas Related to Conjugate Points\label{conj section}}


In this section, we prove a technical estimate, Lemma \ref{Hlemma} (below),
about conjugate points. We then use Lemma \ref{Hlemma} in the next section
to prove Theorem \ref{thm2}.

Throughout this section, suppose $x_{0}$ is a critical point for $\mathrm{%
dist}_{K}$ and $v\in \Uparrow _{x_{0}}^{K}$with $\mathrm{dist}_{K}\left(
x_{0}\right) =c_{0}.$ Let $\gamma _{v}:[0,c_{0}]\longrightarrow M$ be the
segment from $\gamma _{v}(0)=x_{0}$ to $\gamma _{v}(c_{0})=p\in K$ with $%
\gamma _{v}^{\prime }(0)=v$.

Our first lemma generalizes the fact that Jacobi fields are determined by
their boundary values on intervals that are free of conjugate points.

\begin{lemma}
\label{jacobifield} Suppose $w\in S_{x_{0}}$ is orthogonal to $\ker \left(
d\exp _{x_{0}}\right) _{v}.$ Then there is a unique Jacobi field $J_{w}$
along $\gamma _{v}$ so that 
\begin{equation*}
J_{w}\left( 0\right) =w\text{ \ and \ }J_{w}\left( c_{0}\right) =0.
\end{equation*}
\end{lemma}

\begin{proof}
Let $\mathcal{N}$ be the family of Jacobi fields $N$ so that 
\begin{equation*}
N\left( 0\right) =N\left( c_{0}\right) =0.
\end{equation*}%
Let $\mathcal{P}$ be the family of Jacobi fields $P$ so that 
\begin{equation*}
P\left( c_{0}\right) =0\text{ and }g\left( P^{^{\prime }}\left( c_{0}\right)
,N^{^{\prime }}\left( c_{0}\right) \right) =0\text{ for all }N\in \mathcal{N}%
.
\end{equation*}%
We have 
\begin{equation}
\ker \left( d\exp _{x_{0}}\right) _{v}=\left\{ \left. N^{^{\prime }}\left(
0\right) \text{ }\right\vert \text{ }N\in \mathcal{N}\right\} .
\label{character of kern eqn}
\end{equation}

Next we claim that 
\begin{equation}
\left\{ \left. P\left( 0\right) \text{ }\right\vert \ P\in \mathcal{P}%
\right\} \text{ is the orthogonal complement of }\ker \left( d\exp
_{x_{0}}\right) _{v}.  \label{orthog statement}
\end{equation}

First observe that the evaluation map 
\begin{equation}
P\mapsto P\left( 0\right) \text{ is injective.\label{eval statement}}
\end{equation}%
Indeed, if $P\left( 0\right) =0,$ then $P\in \mathcal{N}.$ Combined with $%
g\left( P^{^{\prime }}\left( c_{0}\right) ,N^{^{\prime }}\left( c_{0}\right)
\right) =0$ for all $N\in \mathcal{N},$ it follows that $P\equiv 0,$ and the
evaluation map is indeed injective.

It follows that 
\begin{equation*}
\mathrm{\dim }\left\{ \left. P\left( 0\right) \text{ }\right\vert \ P\in 
\mathcal{P}\right\} =\dim \left( \mathcal{P}\right) =n-1-\mathrm{\dim }\,%
\mathcal{N}.
\end{equation*}

For $P\in \mathcal{P}$ and $N\in \mathcal{N},$ we have 
\begin{eqnarray*}
\frac{\partial }{\partial t}\left( g\left( P,N^{^{\prime }}\right) -g\left(
P^{^{\prime }},N\right) \right) &=&g\left( P,N^{\prime \prime }\right)
-g\left( P^{\prime \prime },N\right) \\
&=&R\left( N,\dot{\gamma}_{v},\dot{\gamma}_{v},P\right) -R\left( P,\dot{%
\gamma}_{v},\dot{\gamma}_{v},N\right) \\
&=&0.
\end{eqnarray*}%
Since $P\left( c_{0}\right) =N\left( c_{0}\right) =0$ and $N\left( 0\right)
=0,$ 
\begin{equation*}
\left. g\left( P,N^{^{\prime }}\right) \right\vert _{c_{0}}=\left. g\left(
P^{^{\prime }},N\right) \right\vert _{c_{0}}=0,
\end{equation*}%
and 
\begin{equation}
\left. g\left( P,N^{^{\prime }}\right) \right\vert _{0}=\left. g\left(
P^{^{\prime }},N\right) \right\vert _{0}=0.  \label{P perp to N' eqn}
\end{equation}%
Together, Equations \ref{character of kern eqn} and \ref{P perp to N' eqn}
give us that $\left\{ P\left( 0\right) |\ P\in \mathcal{P}\right\} $ is
orthogonal to $\ker \left( d\exp _{x_{0}}\right) _{v}.$ On the other hand,
both $\left\{ P\left( 0\right) |\ P\in \mathcal{P}\right\} $ and$\left\{
\ker \left( d\exp _{x_{0}}\right) _{v}\right\} ^{\perp }$ have dimension $%
n-1-\mathrm{\dim }$\thinspace $\mathcal{N},$ thus Statement \ref{orthog
statement} holds.

Now given any $w\bot \ker \left( d\exp _{x_{0}}\right) _{v}$, just choose $%
J_{w}$ to be the unique $P\in \mathcal{P}$ with $P\left( 0\right) =w.$ The
uniqueness of $P$ follows from Statement \ref{eval statement}.
\end{proof}

\begin{lemma}
\label{boundary values lemma}Let $\gamma :\left[ 0,l\right] \longrightarrow
M $ be a segment. For $\varepsilon \in \left( 0,\frac{l}{2}\right] ,$ let 
\begin{equation*}
\mathcal{J}_{\varepsilon }=\left\{ \left. \text{Jacobi fields }J\text{ along 
}\gamma \text{ }\right\vert \text{ }\left\vert J\left( 0\right) \right\vert
=\left\vert J\left( \varepsilon \right) \right\vert =1\right\} .
\end{equation*}%
There is a $B>0$ so that for all $\varepsilon \in \left( 0,\frac{l}{2}\right]
$ and all $J\in \mathcal{J}_{\varepsilon },$%
\begin{equation*}
\left\vert J|_{\left[ 0,\varepsilon \right] }\right\vert \leq B.
\end{equation*}
\end{lemma}

\begin{proof}
Set 
\begin{equation*}
B_{\varepsilon }\equiv \max \left\{ \left. \left\vert J|_{\left[
0,\varepsilon \right] }\right\vert \text{ }\right\vert \text{ }J\in \mathcal{%
J}_{\varepsilon }\right\} .
\end{equation*}%
Then $B_{\varepsilon }$ is a continuous function of $\varepsilon \in \left(
0,\frac{l}{2}\right] .$ So the result follows provided 
\begin{equation*}
\lim_{t\rightarrow 0}\sup_{0<\varepsilon \leq t}B_{\varepsilon }<\infty .
\end{equation*}%
If not, there a sequence of $\varepsilon _{i}\rightarrow 0$ and a sequence
of Jacobi fields $\left\{ J_{\varepsilon _{i}}\right\} _{i=1}^{\infty }$
with 
\begin{eqnarray*}
J_{\varepsilon _{i}} &\in &\mathcal{J}_{\varepsilon _{i}}\text{ and } \\
\left\vert J_{\varepsilon _{i}}\left( s_{i}\right) \right\vert ^{2} &\geq
&i\longrightarrow \infty
\end{eqnarray*}%
for some $s_{i}\in \left[ 0,\varepsilon _{i}\right] .$ Without loss of
generality, we may further assume that 
\begin{equation*}
\left\vert J_{\varepsilon _{i}}\left( s_{i}\right) \right\vert ^{2}\geq
\left\vert J_{\varepsilon _{i}}\left( t\right) \right\vert ^{2}\text{ for
all }t\in \left[ 0,\varepsilon _{i}\right] .
\end{equation*}

By the Mean Value Theorem, 
\begin{equation*}
\frac{d}{dt}\left\vert J_{\varepsilon _{i}}\left( \sigma _{i}\right)
\right\vert ^{2}\geq \frac{\left\vert J_{\varepsilon _{i}}\left(
s_{i}\right) \right\vert ^{2}-1}{s_{i}}\geq \frac{\left\vert J_{\varepsilon
_{i}}\left( s_{i}\right) \right\vert ^{2}-1}{\varepsilon _{i}}
\end{equation*}%
for some $\sigma _{i}\in \left[ 0,s_{i}\right] ,$ and 
\begin{equation*}
\frac{d}{dt}\left\vert J_{\varepsilon _{i}}\left( \tau _{i}\right)
\right\vert ^{2}\leq \frac{1-\left\vert J_{\varepsilon _{i}}\left(
s_{i}\right) \right\vert ^{2}}{\varepsilon _{i}-s_{i}}\leq \frac{%
1-\left\vert J_{\varepsilon _{i}}\left( s_{i}\right) \right\vert ^{2}}{%
\varepsilon _{i}}
\end{equation*}%
for some $\tau _{i}\in \left[ s_{i},\varepsilon _{i}\right] ,$ provided $i$
is large enough so that $\left\vert J_{\varepsilon _{i}}\left( s_{i}\right)
\right\vert ^{2}>1.$

It follows there is a $\kappa _{i}\in \left[ \sigma _{i},\tau _{i}\right] $
such that 
\begin{eqnarray*}
\frac{d}{dt}\frac{d}{dt}\left\vert J_{\varepsilon _{i}}\left( \kappa
_{i}\right) \right\vert ^{2} &\leq &\frac{\frac{1-\left\vert J_{\varepsilon
_{i}}\left( s_{i}\right) \right\vert ^{2}}{\varepsilon _{i}}-\frac{%
\left\vert J_{\varepsilon _{i}}\left( s_{i}\right) \right\vert ^{2}-1}{%
\varepsilon _{i}}}{\tau _{i}-\sigma _{i}} \\
&=&\left( \left\vert J_{\varepsilon _{i}}\left( s_{i}\right) \right\vert
^{2}-1\right) \left( \frac{2}{\varepsilon _{i}}\right) \frac{-1}{\tau
_{i}-\sigma _{i}} \\
&\leq &\left( \left\vert J_{\varepsilon _{i}}\left( s_{i}\right) \right\vert
^{2}-1\right) \left( \frac{2}{\varepsilon _{i}}\right) \frac{-1}{\varepsilon
_{i}} \\
&=&-\frac{2}{\varepsilon _{i}^{2}}\left( \left\vert J_{\varepsilon
_{i}}\left( s_{i}\right) \right\vert ^{2}-1\right) .
\end{eqnarray*}%
On the other hand, 
\begin{eqnarray*}
\left. \frac{d}{dt}\frac{d}{dt}g\left( J_{\varepsilon _{i}},J_{\varepsilon
_{i}}\right) \right\vert _{\kappa _{i}} &=&2g\left( J_{\varepsilon
_{i}}^{\prime \prime }\left( \kappa _{i}\right) ,J_{\varepsilon _{i}}\left(
\kappa _{i}\right) \right) +2g\left( J_{\varepsilon _{i}}^{\prime }\left(
\kappa _{i}\right) ,J_{\varepsilon _{i}}^{\prime }\left( \kappa _{i}\right)
\right) \\
&=&\left. 2R\left( J_{\varepsilon _{i}},\dot{\gamma},\dot{\gamma}%
,J_{\varepsilon _{i}}\right) \right\vert _{\kappa _{i}}+2g\left(
J_{\varepsilon _{i}}^{\prime }\left( \kappa _{i}\right) ,J_{\varepsilon
_{i}}^{\prime }\left( \kappa _{i}\right) \right) .
\end{eqnarray*}%
Combining the previous two displays, 
\begin{eqnarray*}
-\frac{2}{\varepsilon _{i}^{2}}\left( \left\vert J_{\varepsilon _{i}}\left(
s_{i}\right) \right\vert ^{2}-1\right) &\geq &\left. 2R\left( J_{\varepsilon
_{i}},\dot{\gamma},\dot{\gamma},J_{\varepsilon _{i}}\right) \right\vert
_{\kappa _{i}}+2g\left( J_{\varepsilon _{i}}^{\prime }\left( \kappa
_{i}\right) ,J_{\varepsilon _{i}}^{\prime }\left( \kappa _{i}\right) \right)
\\
&\geq &\left. 2R\left( J_{\varepsilon _{i}},\dot{\gamma},\dot{\gamma}%
,J_{\varepsilon _{i}}\right) \right\vert _{\kappa _{i}}.
\end{eqnarray*}%
It follows that 
\begin{equation*}
\left\vert \left. R\left( J_{\varepsilon _{i}},\dot{\gamma},\dot{\gamma}%
,J_{\varepsilon _{i}}\right) \right\vert _{\kappa _{i}}\right\vert \geq 
\frac{1}{\varepsilon _{i}^{2}}\left( \left\vert J_{\varepsilon _{i}}\left(
s_{i}\right) \right\vert ^{2}-1\right) .
\end{equation*}%
But 
\begin{eqnarray*}
\left\vert \left. R\left( J_{\varepsilon _{i}},\dot{\gamma},\dot{\gamma}%
,J_{\varepsilon _{i}}\right) \right\vert _{\kappa _{i}}\right\vert &\leq
&\left\Vert R_{\kappa _{i}}\right\Vert \cdot \left\vert J_{\varepsilon
_{i}}\left( \kappa _{i}\right) \right\vert ^{2} \\
&\leq &\left\Vert R_{\kappa _{i}}\right\Vert \cdot \left\vert J_{\varepsilon
_{i}}\left( s_{i}\right) \right\vert ^{2}.
\end{eqnarray*}%
The previous two displays give 
\begin{eqnarray*}
\frac{1}{\varepsilon _{i}^{2}}\left( \left\vert J_{\varepsilon _{i}}\left(
s_{i}\right) \right\vert ^{2}-1\right) &\leq &\left\Vert R_{\kappa
_{i}}\right\Vert \left\vert J_{\varepsilon _{i}}\left( s_{i}\right)
\right\vert ^{2}\text{ or} \\
\frac{1}{\varepsilon _{i}^{2}}\left( 1-\frac{1}{\left\vert J_{\varepsilon
_{i}}\left( s_{i}\right) \right\vert ^{2}}\right) &\leq &\left\Vert
R_{\kappa _{i}}\right\Vert ,
\end{eqnarray*}%
which yields a contradiction, since as $i\rightarrow \infty ,$ $\left\Vert
R_{\kappa _{i}}\right\Vert \longrightarrow \left\Vert R_{0}\right\Vert $, $%
\varepsilon _{i}\rightarrow 0,$ and\linebreak\ $\frac{1}{\left\vert
J_{\varepsilon _{i}}\left( s_{i}\right) \right\vert ^{2}}\longrightarrow 0.$
\end{proof}

Given $w\in S_{x_{0}}$ and $H\in \mathbb{R},$ we set 
\begin{equation*}
T_{2,H}^{x_{0},w}(t)\equiv \mathrm{dist}_{K}x_{0}-t\cdot \cos
\sphericalangle (w,\Uparrow _{x_{0}}^{K})+\frac{{1}}{{2}}H\cdot t^{2}.
\end{equation*}

\begin{lemma}
\label{Hlemma} Let $c_{0}\equiv \mathrm{dist}_{K}x_{0}$ and let $H$ be:%
\newline
(1) $-g\left( {J_{w}^{\prime }(0),J_{w}(0)}\right) $ if $w\in S_{x_{0}}$ is
orthogonal to $\mathrm{ker}(d\exp _{x_{0}})_{v}$, and $J_{w}$ is as in Lemma %
\ref{jacobifield}. \newline
(2) any number if $w\in S_{x_{0}}$ is not orthogonal to $\mathrm{ker}(d\exp
_{x_{0}})_{v}$.\newline
There exists an interval $[0,m]$, depending on $w$ and $H$, for which 
\begin{equation*}
\mathrm{dist}_{K}\left( {\exp }_{x_{0}}\left( tw\right) \right) \leq
T_{2,H}^{x_{0},w}(t)+o(t^{2}).
\end{equation*}
\end{lemma}

\begin{proof}
Given $w\in S_{x_{0}},$ choose $v\in \Uparrow _{x_{0}}^{K}$ so that 
\begin{equation*}
\sphericalangle \left( w,\Uparrow _{x_{0}}^{K}\right) =\sphericalangle
\left( w,v\right) .
\end{equation*}%
Let $W$ be a vector field along $\gamma _{v}$ with $W(0)=w$ and $W(c_{0})=0,$
and let\linebreak\ ${\tilde{\gamma}:}\left[ 0,c_{0}\right] \times \left(
-\varepsilon ,\varepsilon \right) \longrightarrow M$ be the variation of $%
\gamma _{v}$ obtained by exponentiating $W.$ Then by $1^{st}$--variation,%
\begin{equation*}
\left. \frac{{d\mathrm{Len}(\tilde{\gamma}_{s})}}{{ds}}\right\vert
_{s=0}=-\cos \sphericalangle (w,\Uparrow _{x_{0}}^{K}),{\ }
\end{equation*}%
{and by }$2^{nd}$--{variation,}%
\begin{eqnarray*}
\left. \frac{{d^{2}E(\tilde{\gamma}_{s})}}{{ds^{2}}}\right\vert _{s=0}
&=&\int_{0}^{c_{0}}\left\{ \left\vert W^{\prime }\right\vert ^{2}-R\left( W,%
\dot{\gamma}_{v},\dot{\gamma}_{v},W\right) -\left[ \frac{d}{dt}g\left( W,%
\dot{\gamma}_{v}\right) \right] ^{2}\right\} dt \\
&\leq &\int_{0}^{c_{0}}\left\{ \left\vert W^{\prime }\right\vert
^{2}-R\left( W,\dot{\gamma}_{v},\dot{\gamma}_{v},W\right) \right\} dt.
\end{eqnarray*}%
(See, e.g., \cite{CheegEbin} pages 5, 20.)

Thus it suffices to find a vector field $W$ along $\gamma _{v}$ with $W(0)=w$
and $W(c_{0})=0$ such that $I(W,W)\leq H$, where 
\begin{equation*}
I(W,W)\equiv \int_{0}^{c_{0}}\left\{ \left\vert W^{\prime }\right\vert
^{2}-R\left( W,\dot{\gamma}_{v},\dot{\gamma}_{v},W\right) \right\} dt
\end{equation*}%
is the index form on $\left( W,W\right) .$

Suppose $w$ is orthogonal to $\mathrm{ker}(d\exp _{x_{0}})_{v}$. By Lemma %
\ref{jacobifield}, there is a Jacobi field $J_{w}$ along $\gamma _{v}$ with $%
J_{w}(0)=w$ and $J_{w}(c_{0})=0$. Thus for $H=I(J_{w},J_{w})=-g\left( {\
J_{w}^{\prime }(0),J_{w}(0)}\right) $, the result holds.

Next we consider the special case when $w$ is in $\mathrm{ker}(d\exp
_{x_{0}})_{v}$. In this event, there is a nonzero Jacobi field $J$ along $%
\gamma _{v}$ such that $J(0)=J(c_{0})=0$ and $J^{\prime }(0)=w$. For $%
\varepsilon >0,$ define a vector field $V_{\varepsilon }$ by%
\begin{equation*}
V_{\varepsilon }(t)\equiv 
\begin{cases}
\frac{J(t)}{{|J(\varepsilon )|}} & \text{if }t\in \left[ \varepsilon ,c_{0}%
\right] \vspace*{0.05in} \\ 
Y_{\varepsilon }(t) & \text{if }t\in \lbrack 0,\varepsilon ],%
\end{cases}%
\end{equation*}%
where $Y_{\varepsilon }$ is the Jacobi field with $Y_{\varepsilon
}(\varepsilon )=\frac{{J(\varepsilon )}}{{|J(\varepsilon )|}}$ and $%
Y_{\varepsilon }(0)=J^{\prime }(0)=w$. Then the index form is given by 
\begin{eqnarray*}
I(V_{\varepsilon },V_{\varepsilon }) &=&g\left( {\frac{J^{\prime }(c_{0})}{%
|J(\varepsilon )|},\frac{{J(}c_{0}{)}}{{|J(\varepsilon )|}}}\right) -g\left( 
{\frac{{J^{\prime }(\varepsilon )}}{{|J(\varepsilon )|}},\frac{{%
J(\varepsilon )}}{{|J(\varepsilon )|}}}\right) \\[14pt]
&&+g\left( {Y_{\varepsilon }^{\prime }(\varepsilon ),Y_{\varepsilon
}(\varepsilon )}\right) -g\left( {Y_{\varepsilon }^{\prime
}(0),Y_{\varepsilon }(0)}\right) \\
&=&-\frac{{1}}{{|J(\varepsilon )|^{2}}}g\left( {J^{\prime }(\varepsilon
),J(\varepsilon )}\right) +g\left( {Y_{\varepsilon }^{\prime }(\varepsilon ),%
\frac{{J(\varepsilon )}}{{|J(\varepsilon )|}}}\right) -g\left( {%
Y_{\varepsilon }^{\prime }(0),J^{\prime }(0)}\right) .
\end{eqnarray*}%
The limit of the first term is 
\begin{eqnarray}
\lim_{\varepsilon \rightarrow 0}\frac{{-g\left( {J^{\prime }(\varepsilon
),J(\varepsilon )}\right) }}{{|J(\varepsilon )|^{2}}} &=&-\lim_{\varepsilon
\rightarrow 0}\frac{{g\left( {J^{\prime \prime }(\varepsilon ),J(\varepsilon
)}\right) +g\left( {J^{\prime }(\varepsilon ),J^{\prime }(\varepsilon )}%
\right) }}{{2g\left( {J^{\prime }(\varepsilon ),J(\varepsilon )}\right) }} 
\notag \\[14pt]
&=&-\infty ,  \label{the one that goes to inf inequal}
\end{eqnarray}%
since $J^{\prime }(0)=w\neq 0$ and $J\left( 0\right) =0.$

Given any $C>0,$ it follows that there is an $\varepsilon >0$ so that 
\begin{equation*}
I(V_{\varepsilon },V_{\varepsilon })\leq -C,
\end{equation*}%
provided we can find bounds on 
\begin{equation*}
g\left( {\ Y_{\varepsilon }^{\prime }(\varepsilon ),\frac{{J(\varepsilon )}}{%
{|J(\varepsilon )|}}}\right) \text{ and }g\left( {Y_{\varepsilon }^{\prime
}(0),J^{\prime }(0)}\right) \text{ }
\end{equation*}%
that are independent of $\varepsilon $.

Let $\{E_{i}\}_{i=1}^{n-1}$ be an orthonormal parallel frame for the normal
space of $\gamma _{v}$ with $E_{1}(0)=J^{\prime }(0)$. Write $%
J=\sum_{i=1}^{n-1}f_{i}E_{i}$ where each $f_{i}$ is a smooth function. Since 
$J(0)=0$, $f_{i}(0)=0$ for all $i$. Since $E_{1}(0)=J^{\prime }(0)$, $%
f_{1}^{\prime }(0)=1$ and $f_{i}^{\prime }(0)=0$ for all $i=2,...,n-1$.
Since $J$ is a Jacobi field with $J(0)=0$, 
\begin{equation*}
J^{\prime \prime }(0)=\sum_{i=1}^{n-1}\ f_{i}^{\prime \prime
}(0)E_{i}(0)=-R\left( {J(0),\dot{\gamma}_{v}(0)}\right) \dot{\gamma}%
_{v}(0)=0.
\end{equation*}%
So $f_{i}^{\prime \prime }(0)=0$ for all $i$. It follows that there exists
an interval on which 
\begin{equation*}
f_{1}(t)=t+\mathcal{O}(t^{3})\text{ \ and \ }f_{i}(t)=\mathcal{O}(t^{3})%
\text{ for }i=2,\dots ,n-1.
\end{equation*}%
We use this to approximate $\frac{{J(t)}}{{|J(t)|}}$. First note that 
\begin{equation*}
|J(t)|^{2}=\sum_{i=1}^{n-1}f_{i}^{2}(t)=t^{2}+\mathcal{O}(t^{4})=t^{2}\left( 
{1+\mathcal{O}(t^{2})}\right) .
\end{equation*}%
Taking the square root, we get 
\begin{equation*}
|J(t)|=\sqrt{t^{2}(1+\mathcal{O}(t^{2}))}=t\left( {1+\mathcal{O}(t^{2})}%
\right) =t+\mathcal{O}(t^{3}).
\end{equation*}%
%
%
%
%
%
%
%
%
%
%
%
%
%
%
%
%
%
%
%
%
%
%
%
%
%
%
%
%
%
%
%
%
%
%
%
%
%
%
%
%
%
%
%
%
%
%
%
%
%
%
%
%
%
%
%
%
%
%
%
%
%
%
%
%
%
%
%
%
%
%
%
%
%
%
%
%
%
%
%
%
%
%
%
%
%
%
%
%
%
%
%
%
%
%
%
%
%
%
%
%
%
%
%
%
%
%
%
%
%
%
%
%
%
%
%
%
%
%
%
%
%
%
%
%
%
%
%
%
%
%
%
%
%
%
%
%
%
%
%
%
%
Combining these 
\begin{eqnarray}
\frac{{J(t)}}{{|J(t)|}} &=&\frac{{\sum_{i=1}^{n-1}f_{i}(t)E_{i}(t)}}{{t+%
\mathcal{O}(t^{3})}}  \label{f} \\
&=&\left( 1+\mathcal{O}(t^{2})\right) E_{1}+{\sum_{i=2}^{n-1}\mathcal{O}%
(t^{2})E_{i}(t).}  \notag
\end{eqnarray}

In order to approximate $Y_{\varepsilon }^{\prime }$, we write $%
Y_{\varepsilon }=\sum_{i=1}^{n-1}h_{\varepsilon ,i}E_{i}$, where each $%
h_{\varepsilon ,i}$ is a smooth function that depends on $\varepsilon $. By
Lemma \ref{boundary values lemma}, there is $B>0$ so that for all $%
\varepsilon >0,$ 
\begin{equation}
\left\vert \left. h_{\varepsilon ,i}^{\prime \prime }\right\vert _{\left[
0,\varepsilon \right] }\right\vert =\left\vert \left. g\left( Y_{\varepsilon
}^{\prime \prime },E_{i}\right) \right\vert _{\left[ 0,\varepsilon \right]
}\right\vert =\left\vert \left. R(Y_{\varepsilon },\dot{\gamma}_{v},\dot{%
\gamma}_{v},E_{i})\right\vert _{\left[ 0,\varepsilon \right] }\right\vert
\leq B.  \label{boundong''}
\end{equation}%
Thus 
\begin{equation*}
\left\vert Y_{\varepsilon }^{\prime \prime }\right\vert ^{2}=\left\vert
\sum_{i=1}^{n-1}h_{\varepsilon ,i}^{\prime \prime }E_{i}\right\vert ^{2}\leq
\left( n-1\right) B^{2}.
\end{equation*}%
Since $Y_{\varepsilon }(0)=J^{\prime }(0)$ and $E_{1}(0)=J^{\prime }(0)$, we
have $h_{\varepsilon ,1}(0)=1$ and $h_{\varepsilon ,i}(0)=0$ for $i=2,\dots
,n-1$. Taylor's Theorem combined with Inequality \ref{boundong''} give us an
interval $[0,m]$, independent of $\varepsilon $, on which 
\begin{equation}
h_{\varepsilon ,1}(t)=1+h_{\varepsilon ,1}^{\prime }(0)t+\mathcal{O}(t^{2})
\label{h_1eps eqn}
\end{equation}%
and 
\begin{equation*}
h_{\varepsilon ,i}(t)=h_{\varepsilon ,i}^{\prime }(0)t+\mathcal{O}(t^{2})%
\text{ for }i=2,\dots ,n-1.
\end{equation*}%
Using Equation \ref{f}, we have 
\begin{equation}
\sum_{i=1}^{n-1}h_{\varepsilon ,i}\left( \varepsilon \right) E_{i}\left(
\varepsilon \right) =Y_{\varepsilon }(\varepsilon )=\frac{{J(\varepsilon )}}{%
{|J(\varepsilon )|}}=\left( 1+\mathcal{O}(\varepsilon ^{2})\right) E_{1}+{%
\sum_{i=2}^{n-1}\mathcal{O}(\varepsilon ^{2})E_{i}(\varepsilon ).\label{h_i
epsss eqn}}
\end{equation}

Combining the previous three displays gives 
\begin{equation}
1+h_{\varepsilon ,1}^{\prime }(0)\varepsilon +\mathcal{O}(\varepsilon
^{2})=h_{\varepsilon ,1}(\varepsilon )=1+\mathcal{O}(\varepsilon ^{2})\text{
and\label{h_1 eps ' eqn} }
\end{equation}%
\begin{equation*}
h_{\varepsilon ,i}^{\prime }(0)\varepsilon +\mathcal{O}(\varepsilon
^{2})=h_{\varepsilon ,i}(\varepsilon )=\mathcal{O}(\varepsilon ^{2})\text{
for }i\geq 2.
\end{equation*}%
So 
\begin{equation}
h_{\varepsilon ,i}^{\prime }(0)=\mathcal{O}(\varepsilon )\text{ for all }i.
\label{h_eps'
ineqaul}
\end{equation}%
Combining Equation \ref{h_eps' ineqaul} with $Y_{\varepsilon }^{\prime
}(0)=\sum_{i=1}^{n-1}h_{\varepsilon ,i}^{\prime }(0)E_{i}(0)$ and $%
\left\vert {J^{\prime }(0)}\right\vert =1$ gives 
\begin{equation}
\left\vert g\left( {Y_{\varepsilon }^{\prime }(0),J^{\prime }(0)}\right)
\right\vert ^{2}\leq \left\vert {Y_{\varepsilon }^{\prime }(0)}\right\vert
^{2}\leq \mathcal{O}(\varepsilon ^{2}).  \label{3rdterm}
\end{equation}

Next we estimate $Y_{\varepsilon }^{\prime }(\varepsilon
)=\sum_{i=1}^{n-1}h_{\varepsilon ,i}^{\prime }(\varepsilon
)E_{i}(\varepsilon )$ by bounding $h_{\varepsilon ,i}^{\prime }(\varepsilon
) $. Combining Inequality \ref{boundong''} with the fact that $%
h_{\varepsilon ,i}^{\prime }(0)=\mathcal{O}(\varepsilon )$, we have 
\begin{equation*}
\left\vert h_{\varepsilon ,i}^{\prime }(\varepsilon )\right\vert =\left\vert
h_{\varepsilon ,i}^{\prime }(0)+\int_{0}^{\varepsilon }h_{\varepsilon
,i}^{\prime \prime }(t)dt\right\vert \leq \mathcal{O}(\varepsilon ).
\end{equation*}%
Thus 
\begin{equation}
\left\vert g\left( {Y_{\varepsilon }^{\prime }(\varepsilon ),\frac{%
J(\varepsilon )}{|J(\varepsilon )|}}\right) \right\vert \leq |Y_{\varepsilon
}^{\prime }(\varepsilon )|\leq \mathcal{O}(\varepsilon ).  \label{2ndterm}
\end{equation}

Combining (\ref{the one that goes to inf inequal}), (\ref{3rdterm}), and (%
\ref{2ndterm}), we have 
\begin{eqnarray*}
I(V_{\varepsilon },V_{\varepsilon }) &=&-\frac{{1}}{{|J(\varepsilon )|^{2}}}%
g\left( {\ J^{\prime }(\varepsilon ),J(\varepsilon )}\right) +g\left( {\
Y_{\varepsilon }^{\prime }(\varepsilon ),\frac{J(\varepsilon )}{%
|J(\varepsilon )|}}\right) -g\left( {\ Y_{\varepsilon }^{\prime
}(0),J^{\prime }(0)}\right) \\[14pt]
&\leq &-\frac{{1}}{{|J(\varepsilon )|^{2}}}g\left( {\ J^{\prime
}(\varepsilon ),J(\varepsilon )}\right) +|\mathcal{O}(\varepsilon
)|\longrightarrow -\infty \text{ \ as }\varepsilon \rightarrow 0.
\end{eqnarray*}

To handle the general case, write $w=w_{\mathrm{tang}}+w_{\bot }$ with $w_{%
\mathrm{tang}}\in \mathrm{ker}(d\exp _{x_{0}})_{v}$ and $w_{\bot }\in \left( 
\mathrm{ker}(d\exp _{x_{0}})_{v}\right) ^{\perp }.$ Let $U$ and $J$ be
Jacobi fields along $\gamma _{v}$ with 
\begin{eqnarray*}
U(0) &=&w_{\bot },\text{ }U(c_{0})=0,\text{ } \\
J(0) &=&J(c_{0})=0,\text{ and }J^{\prime }(0)=w_{tang}.
\end{eqnarray*}%
As before, for $\varepsilon >0,$ we define a vector field $V_{\varepsilon }$
by%
\begin{equation*}
V_{\varepsilon }(t)\equiv 
\begin{cases}
\frac{J(t)}{{|J(\varepsilon )|}}\left\vert w_{\mathrm{tang}}\right\vert & 
\text{if }t\in \left[ \varepsilon ,c_{0}\right] \vspace*{0.05in} \\ 
Y_{\varepsilon }(t) & \text{if }t\in \lbrack 0,\varepsilon ],%
\end{cases}%
\end{equation*}%
where $Y_{\varepsilon }$ is the Jacobi field with $Y_{\varepsilon
}(\varepsilon )=\frac{{J(\varepsilon )}}{{|J(\varepsilon )|}}\left\vert w_{%
\mathrm{tang}}\right\vert $ and $Y_{\varepsilon }(0)=J^{\prime }(0)=w_{%
\mathrm{tang}}$. Notice that $\tilde{W}_{\varepsilon }=U+V_{\varepsilon }$
is a vector field along $\gamma _{v}$ that satisfies 
\begin{eqnarray*}
\tilde{W}_{\varepsilon }(0) &=&w\text{ and} \\
\tilde{W}_{\varepsilon }(c_{0}) &=&0.
\end{eqnarray*}%
So given $H\in \mathbb{R},$ it suffices to show that for $\varepsilon $
sufficiently small, $I(\tilde{W}_{\varepsilon },\tilde{W}_{\varepsilon
})\leq H$. Since 
\begin{equation*}
I(\tilde{W}_{\varepsilon },\tilde{W}_{\varepsilon })=I\left( U,U\right)
+2I\left( U,V_{\varepsilon }\right) +I\left( V_{\varepsilon },V_{\varepsilon
}\right) ,
\end{equation*}%
it follows from the first two cases, that it is sufficient to bound $%
I(U,V_{\varepsilon })$ from above by a constant that is independent of $%
\varepsilon .$

Since ${U(c_{0})=0,}$ 
\begin{equation*}
I(U,V_{\varepsilon })=-g\left( \frac{J^{\prime }(\varepsilon )}{{%
|J(\varepsilon )|}}\left\vert w_{\mathrm{tang}}\right\vert {,U(\varepsilon )}%
\right) +g\left( {Y_{\varepsilon }^{\prime }(\varepsilon ),U(\varepsilon )}%
\right) -g\left( {\ Y_{\varepsilon }^{\prime }(0),U(0)}\right) .
\end{equation*}%
Since $U$ does not depend on $\varepsilon $ and is bounded, Inequalities \ref%
{3rdterm} and \ref{2ndterm} give 
\begin{equation}
\left\vert g\left( {\ Y_{\varepsilon }^{\prime }(\varepsilon ),U(\varepsilon
)}\right) -g\left( {Y_{\varepsilon }^{\prime }(0),U(0)}\right) \right\vert
\leq \mathcal{O}(\varepsilon ).  \label{small guys inequal}
\end{equation}

To estimate $g\left( {\frac{J^{\prime }(\varepsilon )}{{|J(\varepsilon )|}}%
\left\vert w_{\mathrm{tang}}\right\vert ,U(\varepsilon )}\right) $, we write 
$J=\sum_{i=1}^{n-1}f_{i}E_{i}.$ As before, there is an interval on which 
\begin{equation*}
f_{1}(t)=|w_{tang}|\cdot t+\mathcal{O}(t^{3})\text{ \ and \ }f_{i}(t)=%
\mathcal{O}(t^{3})\text{ for }i\geq 2.
\end{equation*}%
Now write $U=\sum_{i=1}^{n-1}k_{i}E_{i}$, where each $k_{i}$ is a smooth
function. Since $U(0)\perp J^{\prime }(0)$, we have $k_{1}(0)=0$. Thus $%
k_{1}(t)=\mathcal{O}(t)$ and $k_{i}(t)=h_{i}(0)+\mathcal{O}(t)$ for $i\geq 2$
on a uniform interval. So for sufficiently small $\varepsilon ,$ 
\begin{eqnarray*}
g\left( {\frac{J^{\prime }(\varepsilon )}{{|J(\varepsilon )|}}\left\vert w_{%
\mathrm{tang}}\right\vert ,U(\varepsilon )}\right) &=&\frac{{\left\vert w_{%
\mathrm{tang}}\right\vert }}{{|J(\varepsilon )|}}\sum_{i=1}^{n-1}f_{i}^{%
\prime }(\varepsilon )k_{i}(\varepsilon ) \\
&=&\frac{{\left\vert w_{\mathrm{tang}}\right\vert }}{{|J(\varepsilon )|}}%
\left[ {\left( {|w_{tang}|+\mathcal{O}(\varepsilon ^{3})}\right) \mathcal{O}%
(\varepsilon )+\sum_{i=2}^{n-1}\mathcal{O}(\varepsilon ^{2})\left( {h_{i}(0)+%
\mathcal{O}(\varepsilon )}\right) }\right] \\
&=&\frac{{\mathcal{O}(\varepsilon )}}{{|J(\varepsilon )|}} \\
&\leq &L,
\end{eqnarray*}%
where $L$ is a constant that is independent of $\varepsilon $.

Combined with Inequality \ref{small guys inequal}, 
\begin{eqnarray}
I(U,V_{\varepsilon }) &=&-g\left( {\ \frac{J^{\prime }(\varepsilon )}{%
|J(\varepsilon )|}\left\vert w_{\mathrm{tang}}\right\vert ,U(\varepsilon )}%
\right) +g\left( {Y_{\varepsilon }^{\prime }(\varepsilon ),U(\varepsilon )}%
\right) -g\left( {Y_{\varepsilon }^{\prime }(0),U}\left( 0\right) \right) 
\notag \\
&\leq &L+|\mathcal{O}(\varepsilon )|.  \label{2nd}
\end{eqnarray}
\end{proof}

\section{Critical Points that Impact the Fundamental Group\label{pi_1 sect}}

In this section we restate Theorem \ref{thm2} and outline its proof.\vspace*{%
0.11in}

\noindent \textbf{Theorem D.}\emph{\ Suppose that the critical points for }$%
\mathrm{dist}\left( K,\cdot \right) $\emph{\ are isolated and that for some }%
$c_{0}>0$\emph{\ and all sufficiently small }$\varepsilon >0,$\emph{\ }%
\begin{equation*}
\pi _{1}(B\left( K,c_{0}+\varepsilon \right) ,B\left( K,c_{0}-\varepsilon
\right) )\neq 0.
\end{equation*}

\emph{Then there is a critical point }$x_{0}$\emph{\ for }$\mathrm{dist}%
\left( K,\cdot \right) $\emph{\ with }$\mathrm{dist}\left( K,x_{0}\right)
=c_{0}$\emph{\ so that there are only two minimal geodesics from }$K$\emph{\
to }$x_{0}$\emph{\ that make angle }$\pi $\emph{\ at }$x_{0}$\emph{.
Moreover, the ends of these geodesic segments are not conjugate along the
segments.}\vspace*{0.11in}

\begin{proof}[Outline of proof]
If for all sufficiently small $\varepsilon ,$ $\pi _{1}(B\left(
K,c_{0}+\varepsilon \right) ,B\left( K,c_{0}-\varepsilon \right) )\neq 0,$
Theorem \ref{mainthm} implies there is a critical point $x_{0}$ for $\mathrm{%
dist}_{K}$ of sub-index $1$ with $\mathrm{dist}_{K}\left( x_{0}\right)
=c_{0}.$ It follows that $\Uparrow _{x_{0}}^{K}$ is a pair of antipodal
points, say $v$ and $-v$.

Let $\gamma _{v}$ be the unique geodesic with $\gamma _{v}^{\prime }\left(
0\right) =v.$ It remains to show that $x_{0}$ is neither conjugate to $%
\gamma _{v}\left( c_{0}\right) $ along $\gamma _{v}$ nor to $\gamma
_{-v}\left( c_{0}\right) $ along $\gamma _{-v}.$ Suppose that $x_{0}$ is
conjugate to $\gamma _{v}\left( c_{0}\right) $ along $\gamma _{v}.$ Then $%
\mathcal{K}\equiv \mathrm{ker}(d\exp _{x_{0}})_{v}$ has dimension $\geq 1,$
and $\mathcal{K}^{\bot },$ the orthogonal complement of $\mathcal{K}$, has
dimension $\leq n-2.$ We will show this implies that for all sufficiently
small\emph{\ }$\varepsilon >0,$\emph{\ }%
\begin{equation*}
\pi _{1}(B\left( K,c_{0}+\varepsilon \right) ,B\left( K,c_{0}-\varepsilon
\right) )=0.
\end{equation*}

Suppose $\iota :E^{1}\longrightarrow B\left( K,c_{0}+\varepsilon \right) $
is a $1$--cell with $\iota \left( \partial E^{1}\right) \in B\left(
K,c_{0}-\varepsilon \right) .$ As in the proof of Theorem \ref{mainthm}, we
apply Lemma \ref{localreduction} to a sufficiently small $r<\frac{1}{4}\cdot 
\mathrm{inj}_{x_{0}}$ to deform $\iota $ so that 
\begin{equation*}
\iota \left( E^{1}\right) \subset \left\{ B\left( K,c_{0}-\frac{r}{2}\right)
\cup B(x_{0},r)\right\} .
\end{equation*}

Since%
\begin{eqnarray*}
\dim E^{1}+\dim \mathcal{K}^{\bot } &\leq &1+n-2 \\
&<&n,
\end{eqnarray*}%
we can use transversality, as in Case 2 of the proof of Theorem \ref{mainthm}%
, to move $\iota $ so that 
\begin{equation*}
\iota \left( E^{1}\right) \subset \left\{ B\left( K,c_{0}-\frac{r}{2}\right)
\cup B(x_{0},r)\right\} \setminus \exp _{x_{0}}\left\{ \left. v\in \mathcal{K%
}^{\bot }\text{ }\right\vert \text{ }\left\vert v\right\vert \leq r\right\} .
\end{equation*}

Then, as in Case 2 of the proof of Theorem \ref{mainthm}, we combine the
proof of Lemma \ref{case 2 Lemma} with Part $2$ of Lemma \ref{Hlemma} to
show that if $r$ is sufficiently small, then we can move $\iota $ so that 
\begin{equation*}
\iota \left( E^{1}\right) \subset B\left( K,c_{0}\right) .
\end{equation*}

This contradicts our hypothesis that $\pi _{1}(B\left( K,c_{0}+\varepsilon
\right) ,B\left( K,c_{0}-\varepsilon \right) )\neq 0$.
\end{proof}

\section{Other Versions of the Connectivity Results\label{alt section}}

We close pointing out that our techniques also yield the following
alternative versions of Theorems \ref{mainthm} and \ref{thm2}.

\begin{theorem}
Suppose $x_{0}$ is an isolated critical point for $\mathrm{dist}\left(
K,\cdot \right) $ with $\mathrm{dist}\left( K,x_{0}\right) =c_{0}$ and
sub-index $\lambda .$ Then for all sufficiently small $\varepsilon >0$, the
inclusion $B\left( K,c_{0}-\varepsilon \right) \hookrightarrow B\left(
K,c_{0}\right) \cup B\left( x_{0},\varepsilon \right) $ is $(\lambda -1)$%
-connected.

That is, 
\begin{equation*}
\pi _{i}(B\left( K,c_{0}\right) \cup B\left( x_{0},\varepsilon \right) ,%
\text{ }B\left( K,c_{0}-\varepsilon \right) )=0
\end{equation*}%
for $i=0,1,\ldots ,(\lambda -1)$.
\end{theorem}

\begin{theorem}
Suppose $x_{0}$ is an isolated critical point for $\mathrm{dist}\left(
K,\cdot \right) $ with $\mathrm{dist}\left( K,x_{0}\right) =c_{0}$ and that
for all sufficiently small $\varepsilon >0,$ 
\begin{equation*}
\pi _{1}(B\left( K,c_{0}\right) \cup B\left( x_{0},\varepsilon \right)
,B\left( K,c_{0}-\varepsilon \right) )\neq 0.
\end{equation*}

Then there are only two minimal geodesics from $K$ to $x_{0}$ that make
angle $\pi $ at $x_{0}$. Moreover, the ends of these geodesic segments are
not conjugate along the segments.
\end{theorem}

\end{document}